\documentclass{amsart}

\usepackage{hyperref,amsmath,amssymb,amsthm,newlfont,graphicx,amscd,bbm,enumerate,galois,mathrsfs,xypic,geometry}

\usepackage{color}

\newtheorem{Th}{Theorem}[section]
\newtheorem{Cor}[Th]{Corollary}
\newtheorem{Lem}[Th]{Lemma}
\newtheorem{Prop}[Th]{Proposition}
\newtheorem{Def}{Definition}[section]
\newtheorem{Ex}{Example}[section]
\newtheorem{Conj}{Conjecture}[section]

\newtheorem*{Rem}{Remark}
\numberwithin{equation}{section}

\begin{document}
\newcommand{\xym}[1]{\begin{center}\leavevmode \xymatrix{#1} \end{center}}
\newcommand{\diff}[2]{\frac{\partial #1}{\partial #2}}
\newcommand{\vare}{\varepsilon}
\newcommand{\secp}[1]{\left((u(x)-#1)\partial_x+\frac{1}{2}u^\prime(x)+\frac{\vare^2}{8}\partial_x^3\right)}
\newcommand{\qa}{\alpha}
\newcommand{\qb}{\beta}
\newcommand{\qd}{\delta}
\newcommand{\qg}{\gamma}
\newcommand{\qs}{\sigma}
\newcommand{\qt}{\tau}
\newcommand{\qth}{\theta}
\newcommand{\qe}{\varepsilon}
\newcommand{\qz}{\zeta}
\newcommand{\qp}{\partial}
\newcommand{\qo}{\omega}
\newcommand{\contr}{\mathbin{\raisebox{\depth}{\scalebox{1}[-1]{$\lnot$}}}}
\newcommand{\Qg}{\Gamma}
\newcommand{\ql}{\lambda}
\newcommand{\Qd}{\Delta}
\newcommand{\Qo}{\Omega}
\title{Super tau-covers of bihamiltonian integrable hierarchies}

\author{Si-Qi Liu,  Zhe Wang, Youjin Zhang}
\dedicatory{Dedicated to the memory of Boris Anatol'evich Dubrovin}
\maketitle

\begin{abstract}
We consider a certain super extension, called the super tau-cover, of a bihamiltonian integrable hierarchy which contains the Hamiltonian structures including both the local and non-local ones as odd flows. 
In particular, we construct the super tau-cover of the principal hierarchy associated with an arbitrary Frobenius manifold, and the super tau-cover of the Korteweg-de Vries (KdV) hierarchy. We also show that the Virasoro symmetries of these bihamiltonan integrable hierarchies can be extended to symmetries of the associated super tau-covers.
\end{abstract}
\tableofcontents
\date{\today}

\section{Introduction}
The deep interrelation between 2D topological field theories (TFT) and integrable systems has been a focus of study during the past thirty years, see \cite{buryak2015double, buryak2012polynomial,dijkgraaf1992intersection, dubrovin1992integrable, dubrovin1996geometry, dubrovin1998bihamiltonian, dubrovin2001normal, eguchi1995genus,  givental2005simple, kontsevich1992intersection, witten1990two} and references therein. A class of integrable systems that play important roles in this study consists of bihamiltonian integrable hierarchies of hydrodynamic type and their deformations. Each of these integrable hierarchies of hydrodynamic type is associated to a Frobenius manifold which characterizes a 2D TFT at tree level approximation, and under the assumption of semisimplicity of the Frobenius manifold a certain deformation, called the topological deformation, of the integrable hierarchy of hydrodynamic type controls the full 2D TFT, i.e. the partition function of the 2D TFT is the tau function of a particular solution of the deformed integrable hierarchy. The integrable hierarchy of hydrodynamic type is called the principal hierarchy of the associated Frobenius manifold, and its topological deformation is specified by the condition that the Virasoro symmetries can be represented as linear actions on its tau function. This condition of linearization of the Virasoro symmetries leads to a so called quasi-Miura transformation which transforms the principal hierarchy to its topological deformation. 
On the other hand, the principal hierarchy has a bihamiltonian structure of hydrodynamic type which is given by the flat metric and the intersection form of the Frobenius manifold, and its deformations can also be obtained from deformations of the bihamiltonian structure of hydrodynamic type \cite{dubrovin1996geometry, dubrovin1998bihamiltonian, dubrovin2001normal}. 

As it was proved in \cite{DLZ-1, LZ}, the moduli space of the infinitesimal deformations of semisimple bihamiltonian structures of hydrodynamic type with $n$ components is parametrized by $n$ one-variable functions which are called the central invariants of the deformations, and it was also conjectured that the same moduli space characterizes the full deformations. This conjecture was proved in \cite{liu2013bihamiltonian} for the bihamiltonian structure of the dispersionless KdV hierarchy. For the general case, Carlet, Posthuma and Shadrin proved this conjecture in \cite{carlet2018deformations} by introducing some spectral sequences on the complex constructed in \cite{liu2013bihamiltonian}. For a semisimple Frobenius manifold, the bihamiltonian structure of the associated principal hierarchy is semisimple, and moreover it possesses the so called flat exactness property, a notion which is introduced in \cite{dubrovin2018bihamiltonian}. It was shown in \cite{dubrovin2018bihamiltonian} that for 
the deformations of a flat exact semisimple bihamiltonian structure of hydrodynamic type with constant central invariants, the corresponding deformations of the principal hierarchy possess tau functions. An important question is whether the deformations of the principal hierarchy also 
possess Virasoro symmetries, and if so, how do these Virasoro symmetries act on the tau functions. In particular, we conjecture that when the central invariants of the deformed bihamiltonian structure are equal to $\frac1{24}$, the deformation of the principal hierarchy has Virasoro symmetries, and 
these Virasoro symmetries act linearly on its tau function. Thus in this case the deformed principal hierarchy is conjectured to be equivalent to the topological deformation that is obtained from the principal hierarchy by using the quasi-Miura transformation. 

The main purpose of the present paper is to propose an approach to the study of the question mentioned above. For any given Frobenius manifold, we are to construct a certain super extension of the associated principal hierarchy, in which the bihamiltonian structures of the principal hierarchy are represented by the odd flows. We then construct a tau-cover of this super extension, called the super tau-cover of the principal hierarchy, and show that the Virasoro symmetries of the principal hierarchy can be lifted to symmetries of its super tau-cover. In particular, the odd flows representing the Hamiltonian structures commute with the Virasoro symmetries of the super tau-cover of the principal hierarchy. 

We hope that we could have a better understanding of the relationship between the bihamiltonian structures and the Virasoro symmetries of the deformations of the principal hierarchy by studying deformations of the super tau-cover of the principal hierarchy.
For the example of one-dimensional Frobenius manifold, we give a 
deformation of the super tau-cover of the associated principal hierarchy by using a super extension of the Lax pair of the KdV hierarchy. This deformation is uniquely determined by the condition of linearization of the Virasoro symmetries.

The paper is organized as follows. In Sec.\,\ref{sec2}, we recall the formulation of evolutionary PDEs and their Hamiltonian structures in terms of the infinite jet space of a supermanifold. In Sec.\,\ref{sec3} and Sec.\,\ref{sec4}, we construct the super tau-cover of the principal hierarchy associated with an arbitrary Frobenius manifold and prove the main results of the paper. In Sec.\,\ref{sec5}, we present a deformation of the super tau-cover of the principal hierarchy associated to the 1-dimensional Frobenius manifold which is the super tau-cover of the KdV hierarchy. In Sec.\,\ref{sec6} we give some concluding remarks.
\vskip 1ex

\noindent\textbf{Acknowledgements.}
We would like to thank Boris Dubrovin for very helpful discussions during the summer of 2018 when Y.Z. visited SISSA. Y.Z. also thanks SISSA for the hospitality extended to him during his visit.
This work is supported by NSFC No.\,11771238 and No.\,11725104.


\section{Super variables and Hamiltonian structures}
\label{sec2}
In this section, we recall the formulation of evolutionary PDEs and their Hamiltonian structures in terms of the infinite jet space of a supermanifold, one can refer to \cite{dubrovin2018bihamiltonian, liu2011jacobi, liu2013bihamiltonian} for more details. 

Let $M$ be an $n$-dimensional smooth manifold, then the associated infinite jet bundle $J^\infty(M)$ is the fiber bundle over $M$ with the fiber $\mathbb R^\infty$. If $U\times \mathbb R^\infty$ and $V\times \mathbb R^\infty$ are two local trivializations with charts $(u^\qa;u^{\qa,s})$ and $(v^\qa;v^{\qa,s})$, where $\qa = 1,\cdots, n$ and $s\geq 1$, then the transition functions are given by the chain rule:
\begin{equation*}
v^{\qa,1} = \diff{v^\qa}{u^\qb}u^{\qb,1};\quad v^{\qa,2} = \diff{v^\qa}{u^\qb}u^{\qb,2}+\frac{\qp^2 v^\qa}{\qp u^\qb\qp u^\qg}u^{\qb,1}u^{\qg,1},\cdots.
\end{equation*}
Here and henceforth summations over repeated Greek upper and lower indices are assumed.

Denote by $\hat M$ the super manifold of dimension $(n|n)$ obtained by reversing the parity of fibers of the cotangent bundle $T^*M$. For the associated infinite jet bundle $J^\infty(\hat M)$, we can choose  $\{u^{\qa,s},\qth_\qa^s \mid \qa = 1,2,\cdots,n;s \geq 0\}$ as its local coordinates, here $(u^{\qa,0};\theta^0_\qa)$ are the local coordinates of the base manifold $\hat M$ which we also denote by $(u^\qa;\theta_\qa)$.
We call $\qth^s_\qa$ super variables due to the following anti-symmetric condition:
\begin{equation*}
\qth^s_\qa\qth^t_\qb +\qth^t_\qb\qth^s_\qa = 0.
\end{equation*}

Let us consider the ring of differential polynomials $\mathcal{\hat A}(M)$ which is locally defined by
\begin{equation}\label{zh-3}
C^\infty(U)\otimes \mathbb C[u^{\qa,s},\qth_\qa^t\mid \qa=1,\dots, n, s\ge 1, t\ge 0].
\end{equation}
It has a super gradation defined by the assignment $\deg{\qth_\qa^s} = 1$, $\deg{u^{\qa, s}}=0$ and the set of homogeneous elements of degree $p$ is denoted by $\mathcal{\hat A}^p(M)$. The global vector field 
\begin{equation}\label{zh-4}
\partial = \sum_{s\geq 0}u^{\qa,s+1}\diff{}{u^{\qa,s}}+\qth_\qa^{s+1}\diff{}{\qth_\qa^s}
\end{equation}
on $J^\infty(\hat M)$ induces a derivation on $\mathcal{\hat A}(M)$.
We will also use the notation $\qp_x$ to denote $\qp$ when we consider 
evolutionary PDEs. The quotient space 
\begin{equation}\label{zh-5}
\mathcal{\hat F}(M):= \mathcal{\hat A}(M)/\partial\mathcal{\hat A}(M)
\end{equation}
also admits a super gradation induced from that of $\mathcal{\hat A}(M)$, and we denote the set of homogeneous elements of degree $p$ by $\mathcal{\hat F}^p(M)$. For an element $f\in\mathcal{\hat A}(M)$ we denote its image in $\mathcal{\hat F}(M)$ by $\int f$ and we call it a local functional of $\hat M$. An important fact is that $\mathcal{\hat F}(M)$ admits a graded Lie algebra structure defined by the Schouten-Nijenhuis bracket
\begin{equation*}
[P,Q] = \int \left(\frac{\qd P}{\qd\qth_\qa}\frac{\qd Q}{\qd u^\qa}+(-1)^p\frac{\qd P}{\qd u^\qa}\frac{\qd Q}{\qd \qth_\qa}\right),\quad \forall P\in \mathcal{\hat F}^p(M),\, \forall Q\in \mathcal{\hat F}^q(M).
\end{equation*}
Here the variational derivatives are defined by
\begin{equation*}
\frac{\qd P}{\qd u^\qa} = \sum_{s\geq 0}(-\qp)^s\diff{\tilde P}{u^{\qa,s}},\quad \frac{\qd P}{\qd \qth_\qa} = \sum_{s\geq 0}(-\qp)^s\diff{\tilde P}{\qth_\qa^s},
\end{equation*}
with $\tilde P \in \mathcal{\hat A}(M)$ being any lift of $P\in\mathcal{\hat F}(M)$. One can refer to \cite{liu2011jacobi} for more differential operators defined on $\mathcal{\hat F}$ and some useful identities satisfied by these operators. 

To each local functional 
$X\in\mathcal{\hat F}^1(M)$ we can associate a system of evolutionary PDEs of the form
\[\frac{\qp u^\qa}{\qp t}=X^\qa,\quad \qa=1,\dots,n,\]
here $X^\qa$ are given by the coefficients of $X$ represented in the form
$X=\int X^\qa \qth_\qa$ with the replacement $u^{\qa,s}\mapsto 
\qp_x^s u^\qa$.
Thus we also call $X$ an evolutionary PDE. We call $X$ a Hamiltonian evolutionary PDE if there exist $P\in \mathcal{\hat F}^2(M)$ and 
$H\in\mathcal{\hat F}^0(M)$ which satisfy the following conditions:
\[ X=-[H, P],\quad [P, P]=0.\]
Here $P$ and $H$ are called the Hamiltonian structure and the Hamiltonian
of $X$ respectively. We can represent $P$ and $H$ in the form
\[ P=\frac12 \int\sum_{s\geq 0} P_s^{\qa\qb}\qth_\qa\qth_\qb^s,\quad 
H=\int h,\]
then the Hamiltonian evolutionary PDE can be represented as
\[\frac{\qp u^\qa}{\qp t}=\mathcal{P}^{\qa\qb}\frac{\delta H}{\delta u^\qb}\quad 
\textrm{with}\quad \mathcal{P}^{\qa\qb}=\sum_{s\ge 0} P^{\qa\qb}_s\qp_x^s,\quad
\frac{\delta H}{\delta u^\qb}=\sum_{s\ge 0} (-\qp_x)^s\frac{\qp h}{\qp u^{\qb,s}}.\]
The differential operator $\mathcal P = (\mathcal{P}^{\qa\qb})$ is called the Hamiltonian operator associated with the Hamiltonian structure $P$. The evolutionary PDE $X$ is called a bihamiltonian system if there exist
$P_0, P_1\in \mathcal{\hat F}^2(M)$ and $H, G\in \mathcal{\hat F}^0(M)$ such that
\[X=-[H, P_0]=-[G, P_1],\quad [P_0, P_0]=[P_1, P_1]=[P_0, P_1]=0.\]
Under some additional conditions, the bihamiltonian structure of the evolutionary PDEs provides a bihamiltonian recursion relation which 
generates an infinite sequence of conserved quantities. The corresponding Hamiltonian flows, together with the originally given one, are mutually commutative, so we usually call such a family of Hamiltonian flows a bihamiltonian integrable hierarchy. 

\begin{Ex}
Let $M$ be a smooth manifold of one dimension, and $(u;\theta)$ be a system of local coordinates of $\hat{M}$. Denote by $\{\qp_x^s u,\ \theta^s\mid s\ge 0\}$ the local coordinates of $J^\infty(\hat{M})$ with $(\qp_x^0 u; \theta^0)=(u;\theta)$. Consider the local functionals
\begin{equation*}
X_0 = \int u_x\qth,
\quad X_n=\frac{2^n}{(2n+1)!!}\int \left(\mathcal{R}^n u_x\right)\theta,\quad n\ge 2,
\end{equation*}
where 
\[\mathcal{R}=\frac{\qe^2}8 \qp_x^2+u+\frac12 u_x\qp_x^{-1}.\]
These local functionals correspond to the KdV hierarchy
\[\frac{\qp u}{\qp t_0}=u_x,\quad \frac{\qp u}{\qp t_1}=u u_x+\frac{\qe^2} {12}u_{xxx},\dots\]
which has a bihamiltonian structure given by the following local functionals 
\begin{equation}
\label{temp}
P_0 = \frac{1}{2}\int \qth\qth^1,\quad P_1 = \frac{1}{2}\int \left(u\qth\qth^1+\frac{\qe^2}{8}\qth\qth^3\right).
\end{equation}
The KdV hierarchy can be represented in terms of bihamiltonian flows as follows:
\begin{equation}\label{zh-1}
\frac{\qp u}{\qp t_p}=\mathcal{P}_0\frac{\delta H_p}{\delta u}
=\left(p+\frac12\right)^{-1}\mathcal{P}_1\frac{\delta H_{p-1}}{\delta u},\quad p\ge 0,
\end{equation}
where the Hamiltonian operators are given by 
\begin{equation}\label{kdv-bho}
\mathcal{P}_0=\qp_x,\quad \mathcal{P}_1=u\qp_x+\frac12 u_x+\frac{\qe^2}8 \qp_x^3,
\end{equation}
and the Hamiltonians can be obtained from
\[H_{-1}=\int u,\quad H_0=\int\left(\frac12 u^2+\frac{\qe^2}{12} u_{xx}\right)\]
by using the bihamiltonian recursion relation given in \eqref{zh-1}.
\end{Ex}

For a local functional $P\in\mathcal{\hat F}^p(M)$, we define a vector field on $J^\infty(\hat M)$ by
\begin{equation}\label{zh-2}
D_P = \sum_{s\geq 0}\partial^s\left(\frac{\qd P}{\qd \qth_\qa}\right)\diff{}{u^{\qa,s}}+(-1)^p\partial^s\left(\frac{\qd P}{\qd u^\qa}\right)\diff{}{\qth_\qa^s}.
\end{equation}
The following relations are satisfied by any given local functionals $P\in \mathcal{\hat F}^p(M)$ and $Q\in\mathcal{\hat F}^q(M)$:
\begin{align}
&[P,Q] = \int D_P(\tilde Q),\\
&(-1)^{p-1}D_{[P,Q]} = D_P\comp D_Q-(-1)^{(p-1)(q-1)}D_Q\comp D_P,\label{ham-vect}
\end{align}
where $\tilde Q \in \mathcal{\hat A}(M)$ is a lift of $Q\in\mathcal{\hat F}(M)$. Motivated by the form of the vector field $D_P$ given in \eqref{zh-2}, we associate to $P$ a system of evolutionary partial differential equations as follows:
\begin{equation*}
\diff{u^\qa}{t_P} = \frac{\qd P}{\qd\qth_\qa},\quad
\diff{\qth_\qa}{t_P} = (-1)^p\frac{\qd P}{\qd u^\qa},\quad \qa=1,\dots,n.
\end{equation*}
We will also denote this flow by $\frac{\qp}{\qp t_P}=D_{P}$. If we have two commutative local functionals $P\in\mathcal{\hat F}^p(M), Q\in\mathcal{\hat F}^q(M)$ with $[P,Q] = 0$, then from \eqref{ham-vect} it follows that the flows $\diff{}{t_P}$ and $\diff{}{t_Q}$ commute with each other.

In particular, for a bihamiltonian integrable hierarchy with bihamiltonian structure $P_0, P_1\in\mathcal{\hat F}^2(M)$, we have the flows
\begin{equation}\label{zh-7}
\diff{u^\qa}{\qt_i} = \frac{\qd P_i}{\qd\qth_\qa},\quad
\diff{\qth_\qa}{\qt_i} = \frac{\qd P_i}{\qd u^\qa},\quad i = 0,1.
\end{equation}
Denote by $\mathcal{P}_0=(\mathcal{P}_0^{\qa\beta})$, $\mathcal{P}_1=(\mathcal{P}_1^{\qa\beta})$ the Hamiltonian operators associated to $P_0, P_1$ respectively, then the flows $\frac{\qp u^\qa}{\qp \tau_0}$ and $\frac{\qp u^\qa}{\qp \tau_1}$ satisfy the recursion relation
\[\frac{\qp u^\qa}{\qp \tau_1}=\mathcal{P}_1^{\qa\beta}\theta_\beta=\left(\mathcal{P}_1\comp\mathcal{P}_0^{-1}\right)^\qa_\gamma\mathcal{P}_0^{\gamma\beta}\theta_\beta =\left(\mathcal{P}_1\comp\mathcal{P}_0^{-1}\right)^\qa_\gamma\frac{\qp u^\gamma}{\qp \tau_0}.\]
In order to study the Virasoro symmetries of the bihamiltonian integrable hierarchy, we need to continue the above recursion procedure to 
consider flows of the form
\[\frac{\qp u^\qa}{\qp \tau_p}=\left(\mathcal{P}_1\comp\mathcal{P}_0^{-1}\right)^\qa_\gamma\frac{\qp u^\gamma}{\qp \tau_{p-1}},\quad p\ge 2.\]
However, these flows are in general nonlocal, namely they can not be represented in terms of local functionals of $\hat{M}$. Due to this reason, we introduce a new family of super variables 
\begin{equation}\label{zh-6}
\{\qs_{\qa,k}^s\mid\qa = 1,\cdots,n;\ k\geq 0,s\geq 0\}
\end{equation}
with $\qs_{\qa,0}^s = \qth_\qa^s$, and we will also denote $\qs_{\qa,k}^0$ by $\qs_{\qa,k}$. 
We enlarge the ring $\mathcal{\hat A}(M)$ to include the new super variables by 
replacing \eqref{zh-3} with
\begin{equation}
C^\infty(U)\otimes \mathbb C[u^{\qa,s}, \qs^t_{\qa,k}\mid \qa=1,\dots, n, s\ge 1,k\ge 0,  t\ge 0],
\end{equation}
and replace the vector field $\partial$ defined in \eqref{zh-4} by
\begin{equation}
\label{dx}
\partial = \sum_{s\geq 0}u^{\qa,s+1}\diff{}{u^{\qa,s}}+\sum_{s,k\geq 0}\qs_{\qa,k}^{s+1}\diff{}{\qs_{\qa,k}^s}.
\end{equation}
We also require the super variables $\sigma_{\alpha,k}^s$ satisfy the following relations:
\begin{equation}\label{gen-biham-rec}
\mathcal{P}^{\qa\qb}_{0}\qs_{\qb,k+1} = \mathcal{P}^{\qa\qb}_{1}\qs_{\qb,k},\quad \qa = 1,\cdots, n.
\end{equation} 
Then we have the new quotient space $\mathcal{\hat F}(M)$ defined as in \eqref{zh-5}, and we still call its elements local functionals of $\hat{M}$.

From the equation \eqref{zh-7}, we see that the Hamiltonian structure $P_0$ and $P_1$ correspond to the odd flow
$\frac{\qp u^{\qa}}{\qp \tau_0}$ and  $\frac{\qp u^{\qa}}{\qp \tau_1}$, so we can regard the higher odd flows
$\frac{\qp u^{\qa}}{\qp \tau_p}$ as higher nonlocal Hamiltonian structures. However, it is not easy to define
variational derivatives on the new quotient space $\mathcal{\hat F}(M)$ so that we can write them as derivations $D_{P_n}$ for certain
bivectors $P_n$.


\section{Super extension of the principal hierarchy of a Frobenius manifold}
\label{sec3}
In this section we construct the super extension of the principal hierarchy of a Frobenius manifold. We first recall the basic facts about Frobenius manifold and the associated principal hierarchy. The main references for the definitions and properties of Frobenius manifolds are \cite{dubrovin1996geometry,dubrovin1999painleve}, and for the principal hierarchies one can refer to \cite{dubrovin2001normal}.

A (complex) Frobenius manifold $M$ is a complex analytic manifold together with an analytic family of Frobenius algebra structures on the tangent bundle $TM$. Such a manifold must be equipped with a flat (pseudo-Riemannian) metric $\eta$ and an Euler vector field $E$. Assuming $\mathrm{dim} M = n$, then locally a Frobenius manifold can be characterized by a solution $F(v^1,\cdots,v^n)$ of the WDVV associativity equations \cite{dijkgraaf1991topological, witten1990structure}, here $v^1,\dots, v^n$ are flat coordinates of the flat metric $\eta$. In terms of the function $F$, which is called the potential of the Frobenius manifold, the components of the flat metric have the expressions
\begin{equation*}
\eta_{\qa\qb} =\qp_\qa\qp_\qb\qp_1 F(v),\quad \qp_\qa = \diff{}{v^\qa},
\end{equation*}
and the Frobenius algebra structure is given by
\begin{equation*}
\qp_\qa\cdot\qp_\qb = c_{\qa\qb}^{\qg}\qp_\qg,\quad c_{\qa\qb}^{\qg} = \eta^{\qg\qz}\qp_\qz\qp_\qa\qp_\qb F(v),
\end{equation*}
here $(\eta^{\qa\qb})=(\eta_{\qa\qb})^{-1}$. The potential $F$ satisfies the quasi-homogeneous condition
\[ E(F)=(3-d)F+\frac12 A_{\qa\qb} v^\qa v^\qb+B_\qa v^\qa+C,\]
where the Euler vector field $E$ has the form
\begin{equation*}
E = \sum_{\qa=1}^n \left(\left(1-\frac{d}{2}-\mu_\qa\right)v^\qa+r_\qa\right)\qp_\qa,
\end{equation*}
$A_{\qa\qb}, B_\qa, C, d$ are some constants, and $d$ is called the charge of the Frobenius manifold. The diagonal matrix $\mu = \mathrm{diag}(\mu_1,\cdots,\mu_n)$ is part of the isomonodromy data of $M$ where we will always take $\mu_1 = -d/2$. The tensors $\eta_{\qa\qb}$ and $c_{\qa\qb}^\qg$ are required to be homogeneous along the Lie derivative of the Euler vector field:
\begin{equation*}
\mathcal{L}_E \eta_{\qa\qb} = (2-d)\eta_{\qa\qb},\quad \mathcal L_E c^\qg_{\qa\qb} = c^\qg_{\qa\qb}.
\end{equation*}
It is then easy to derive the following identities:
\begin{equation}
\label{mu-eta}
(\mu_\qa+\mu_\qb)\eta_{\qa\qb} = 0,\quad\forall\,\qa,\qb = 1,\cdots,n.
\end{equation}
\begin{equation}
\label{c-hom}
E(c_{\qa\qb}^\qg) = (\mu_\qa+\mu_\qb-\mu_\qg-\mu_1)c_{\qa\qb}^\qg,\quad\forall\,\qa,\qb,\qg = 1,\cdots,n.
\end{equation}

The intersection form of the Frobenius manifold $M$ is defined by
\begin{equation*}
g(\qo_1,\qo_2) = E\contr(\qo_1\cdot\qo_2),\quad \qo_1,\qo_2\in T^*M,
\end{equation*}
where the Frobenius algebra structure on the cotangent bundle $T^*M$ is defined via the isomorphism between $TM$ and $T^*M$ induced by the flat metric $\eta$. In terms of the flat coordinates $v^1,\dots,v^n$, the components of $g$ are given by 
\[g^{\qa\qb} = E^\qe c^{\qa\qb}_\qe,\] 
with $c^{\qa\qb}_\qe = \eta^{\qa\qz}c_{\qz\qe}^\qb$. It was shown in \cite{dubrovin1996geometry} that $(\eta_{\qa\qb})$ and $(g_{\qa\qb})=(g^{\qa\qb})^{-1}$ form a flat pencil of metrics. By using the notation of the last section, we
formulate this fact as the following theorem.
\begin{Th}[\cite{dubrovin1996geometry}]
Let $\Qg_{\qa\qb}^\qg$ be the Christoffel symbols of the Levi-Civita connection of $(g_{\qa\qb})$, and denote $\Qg^{\qa\qb}_\qg = -g^{\qa\qe}\Qg_{\qe\qg}^\qb$. Then the local functionals $P_0,P_1\in\mathcal{\hat F}^2(M)$ defined by
\begin{align}
\label{1ham}
P_0 &= \frac{1}{2}\int \eta^{\qa\qb}\qth_\qa\qth_\qb^1,\\
\label{2ham}
P_1 &= \frac{1}{2}\int g^{\qa\qb}\qth_\qa\qth_\qb^1+\Qg_{\qg}^{\qa\qb}v^{\qg,1}\qth_\qa\qth_\qb
\end{align}
form an exact bihamiltonian structure, i.e. they satisfy the 
relations
\begin{align*}
&[P_i,P_j] = 0,\quad i,j = 0,1;\\
& P_0 = [X_1, P_1]\  \textrm{with}\ X_1=\int\theta_1.
\end{align*}
\end{Th} 

The principal hierarchy of the Frobenius manifold $M$ is an integrable hierarchy with bihamiltonian structure $P_0, P_1$, and Hamiltonians $H_{\qa, p}$ defined in terms of the flat coordinates of the deformed flat connection of $M$. Recall that on $M\times\mathbb C^*$ a flat connection $\tilde\nabla$ is constructed in \cite{dubrovin1996geometry} as follows:
\begin{equation*}
\tilde \nabla_XY = \nabla_XY+z X\cdot Y, 
\ \ 
\tilde\nabla_{\qp_z}X = \qp_zX+E\cdot X-\frac{1}{z}\mu X,\ \ \tilde\nabla_{\qp_z}\qp_z = \tilde \nabla_X\qp_z = 0,
\end{equation*}
here $\nabla$ is the Levi-Civita connection of the flat metric $\eta$, and $X, Y$ are vector fields on $M\times \mathbb C^*$ whose $\qp_z$ components being zero. This deformed flat connection has a system of flat coordinates of the form
\begin{equation*}
\left(\tilde v_1(v,z),\cdots,\tilde v_n(v,z)\right) = (h_1(v,z),\cdots,h_n(v,z))z^\mu z^R,
\end{equation*}
where the matrices $R$ and $\mu$ are the isomonodromy data of the Frobenius manifold $M$ at $z = 0$, and the functions $h_{\qa}(v,z)$ are analytic at $z = 0$ with the series expansions 
\begin{equation}
\label{hamil}
h_\qa(v,z) = \sum_{p\geq 0}h_{\qa,p}(v)z^p.
\end{equation}
The functions $h_{\qa, p}(v)$ satisfy the relations
\begin{align}
\label{hamil-rec}
&\qp_\qa\qp_\qb h_{\qg,p+1} = c^\ql_{\qa\qb}\qp_\ql h_{\qg,p},\\
\label{hamil-ini}
&h_{\qa,0} = \eta_{\qa\qg}v^\qg,\quad \qp_1h_{\qa,p+1} = h_{\qa,p},
\end{align}
and the quasi-homogeneous condition
\begin{equation}
\label{homog}
E(\qp_\qb h_{\qa,p}) = (p+\mu_\qa+\mu_\qb)\qp_\qb h_{\qa,p}+\sum_{k=1}^p(R_k)^\qg_\qa\qp_\qb h_{\qg,p-k},
\end{equation}
where the matrices $R_k$ is determined by the decomposition $R = R_1+\cdots+R_m$ satisfying certain conditions, see \cite{dubrovin1996geometry} for details. The flat coordinates of the deformed flat connection are also required to satisfy the normalization condition 
\begin{equation}
\label{norm}
\langle\nabla h_\qa(v,z),\nabla h_\qb(v,-z)\rangle = \eta_{\qa\qb}.
\end{equation}

We are now ready to present the definition of the principal hierarchy.

\begin{Def}
The principal hierarchy associated to an $n$-dimensional Frobenius manifold $M$ is given by the following commuting flows:
\begin{equation}
\label{ph}
\diff{v^\qa}{t^{\qb,p}} =\eta^{\qa\qg}\frac{\qp}{\qp x}\left(\frac{\qp h_{\qb, p+1}}{\qp v^\qg}\right),\quad \qa, \qb=1,\dots, n,\, p\ge 0.
\end{equation}
\end{Def}
From \eqref{hamil-rec}--\eqref{homog} it follows that the flows of the principal hierarchy are bihamiltonian systems with Hamiltonian structures
$P_0, P_1$ given in \eqref{1ham} and \eqref{2ham}, they can be represented by the local functionals 
\[X_{\qa,p} = -[H_{\qa,p},P_0]\in\mathcal{\hat F}^1(M)\ \textrm{with}\ H_{\qa,p} = \int h_{\qa,p+1},\]
and satisfy the following bihamiltonian recursion relation
\[[H_{\qa,p-1},P_1]=\left(p+\frac12+\mu_\qa\right)[H_{\qa,p},P_0]+\sum_{k=1}^p \left(R_k\right)^\qg_\qa [H_{\qg, p-k},P_0].\]
The Hamiltonian densities $h_{\qa,p}$ also satisfy the tau-symmetry condition
\[\frac{\qp h_{\qa,p}}{\qp t^{\beta,q}}=\frac{\qp h_{\beta,q}}{\qp t^{\qa,p}},
\quad \qa, \beta=1,\dots, n; \, p, q\ge 0.\]

Following the discussion in Sec.\,\ref{sec2}, we now introduce a set of super variables $\qs_{\qa,k}^s$ by using the bihamiltonian structure 
$(P_0, P_1)$. The relations given in \eqref{gen-biham-rec} can be written as
\begin{equation}
\label{biham-rec}
\eta^{\qa\qb}\qs_{\qb,k+1}^1 = g^{\qa\qb}\qs_{\qb,k}^1+\Qg_\qg^{\qa\qb}v^{\qg,1}\qs_{\qb,k},\quad \qa=1,\dots, n;\, k\geq 0.
\end{equation}
Let us first extend the action of the flows of the principal hierarchy \eqref{ph} to these 
super variables. By using the vector field $D_{X_{\qb, p}}$ defined on $J^\infty(\hat M)$ by the formula \eqref{zh-2}, we can represent the principal hierarchy \eqref{ph} together with the action of the flow $\frac{\qp}{\qp t^{\qb,p}}$ to the super variable $\sigma_{\qa,0}=\theta_\qa$ as follows:
\begin{equation}
\frac{\qp v^\qa}{\qp t^{\qb, p}}=\frac{\delta X_{\qb, p}}{\delta\theta_\qa},\quad\frac{\qp \theta_\qa}{\qp t^{\qb, p}}=-\frac{\delta X_{\qb, p}}{\delta v^\qa}.
\end{equation}
The evolution of $v^\qa, \theta_\qa$ along the flows $\frac{\qp}{\qp t^{\qb,p}}$ and the recursion relation \eqref{biham-rec} lead to the following lemma. 
\begin{Lem}
The following evolutionary equations are compatible with the recursion relation \eqref{biham-rec}:
\begin{equation}
\label{seg-t-p}
\diff{\qs_{\qa,k}}{t^{\qb,p}} = \eta^{\qg\qe}\frac{\qp^2 h_{\qb, p+1}}{\qp v^\qa\qp v^\qe}\qs_{\qg,k}^1,\quad \qa, \qb=1,\dots,n;\, k, p\ge 0.
\end{equation}
\end{Lem}
\begin{proof}
By compatibility, we mean that
\begin{equation*}
\eta^{\qa\qb}\diff{\qs_{\qb,k+1}^1}{t^{\qe,p}} = \diff{}{t^{\qe,p}}(g^{\qa\qb}\qs_{\qb,k}^1+\Qg_\qg^{\qa\qb}v^{\qg,1}\qs_{\qb,k}),\quad \qa,\qe = 1,\cdots,n.
\end{equation*}
This can be checked directly by using the following formulae:
\begin{align}
\label{gam-v}
&g^{\qa\qb} = E^\qe c^{\qa\qb}_\qe,\quad \Qg^{\qa\qb}_\qg = \left(\frac{1}{2}-\mu_\qb\right)c^{\qa\qb}_\qg;\\
\label{g-gam}
&\qp_\qg g^{\qa\qb} = \Qg_\qg^{\qa\qb}+\Qg_\qg^{\qb\qa}.
\end{align}
The lemma is proved.
\end{proof}
\begin{Rem}
From \eqref{ph}, \eqref{seg-t-p} and \eqref{hamil-ini} it follows that the vector field $\diff{}{t^{1,0}}$ coincides with the vector field $\qp$ that is defined in \eqref{dx}. For this reason, we will identify $t^{1,0}$ with the spatial variable $x$. We will also use prime to denote the derivative with respect to $t^{1,0}$.
\end{Rem}

Now let us consider the flows \eqref{zh-7} associated to the local functionals $P_0, P_1$ defined in \eqref{1ham} and \eqref{2ham} respectively. By using the formulae \eqref{gam-v}, \eqref{g-gam}  we have
\begin{align*}
&\frac{\qp v^\qa}{\qp\qt_0}=\eta^{\qa\qg}\theta_\qb^1,\quad 
\frac{\qp \theta_\qa}{\qp\qt_0}=0;\\
&\frac{\qp v^\qa}{\qp\qt_1}=g^{\qa\qb} \theta_\qb^1+\Gamma^{\qa\qb}_\qg v^{\qg,1}\theta_\qb, 
\quad \frac{\qp \theta_\qa}{\qp\qt_1}=\Gamma^{\qg\qb}_{\qa}\theta_\qb\theta_\qg^1.
\end{align*}
These flows satisfy the following recursion relation:
\[\frac{\qp v^\qa}{\qp\qt_1}=\mathcal{R}^\qa_\qb\left(\frac{\qp v^\qb}{\qp\qt_0}\right)=\eta^{\qa\qg}\sigma^1_{\qg,1},\quad
\mathcal{R}=\mathcal{P}_1\circ  \mathcal{P}_0^{-1}.
\]
Motivated by this recursion relation and the defining relation \eqref{biham-rec} for the super variables $\qs_{\qb, k}$, we introduce the following flows
associated to the nonlocal Hamiltonian structures $\mathcal{P}_m=\mathcal{R}^m \mathcal P_0\,(m\ge 2)$:
\begin{align}
&\diff{v^\qa}{\qt_m} = \eta^{\qa\qb}\qs_{\qb,m}^1,\quad \qa=1,\dots,n;\, m\ge 0.\label{v-tau}\\
&\diff{\qs_{\qa,k}}{\qt_m} = -\diff{\qs_{\qa,m}}{\qt_k} = \Qg^{\qg\qb}_\qa\sum_{i = 0}^{m-k-1}\qs_{\qb,k+i}\qs_{\qg,m-i-1}^1,\quad\qa=1,\dots,n;\,  0\le k\leq m.\label{seg-tau-p}
\end{align}

\begin{Lem}\label{zh-9}
The flows \eqref {v-tau} commute with the ones given in \eqref{ph}.
\end{Lem}
\begin{proof}
It is a straightforward calculation, by using \eqref{ph} and \eqref{seg-t-p}, to verify
\[\diff{}{t^{\qe,p}}\diff{v^\qa}{\qt_m}=\diff{}{\qt_m}\diff{v^\qa}{t^{\qe,p}}.\]
The lemma is proved.
\end{proof}

\begin{Lem}
\label{Z-lem}
The flows \eqref{seg-tau-p} are compatible with the recursion relation \eqref{biham-rec}.
\end{Lem}
\begin{proof}
Introduce the generating functions
\begin{equation}
c_{\qa}(\ql) = -\sum_{m\geq 0}\qs_{\qa,m}\ql^{-m-1},\quad C^m_\qa(\ql) = (\ql^mc_{\qa}(\ql))_-,
\end{equation}
here the subscript ``$-$'' means taking the negative part of the Laurent series with respect to $\ql$. 
By using these generating functions, the flows \eqref{seg-tau-p} can be written as:
\begin{equation}\label{zh-8}
\diff{c_\qa(\ql)}{\qt_m} = \Qg^{\qg\qb}_\qa\left(C^m_\qb(\ql)c_\qg(\ql)^\prime+c_\qb(\ql)C^m_\qg(\ql)^\prime-(\ql^mc_\qb(\ql)c_\qg(\ql)^\prime)_-\right).
\end{equation}
The recursion relation \eqref{biham-rec} gives the following equation satisfied by the generating series:
\begin{equation}
\label{gen-eq}
(g^{\qa\qb}-\ql\eta^{\qa\qb})c_\qb(\ql)^\prime+\Qg_\qg^{\qa\qb}v^{\qg,1}c_\qb(\ql) = \eta^{\qa\qb}\qs_{\qb,0}^1.
\end{equation}

In order to show the compatibility of the flows \eqref{seg-tau-p} with \eqref{gen-eq}, we consider the compatibility of the modified flow
\begin{equation}
\label{temp2-1}
\diff{c_\qa(\ql)}{\tilde\qt_m} = \Qg^{\qg\qb}_\qa\left(C^m_\qb(\ql)c_\qg(\ql)^\prime+c_\qb(\ql)C^m_\qg(\ql)^\prime-(\ql^mc_\qb(\ql)c_\qg(\ql)^\prime)_-\right)+Z_\qa(\ql),
\end{equation}
with \eqref{zh-8}, here $Z_\qa(\ql)= \sum_{k\geq 0}z_{\qb,k}\ql^{-k-1}$ is a Laurent series
with coefficients belonging to $\mathcal{\hat A}(M)$.
Acting $\diff{}{\tilde\qt_m}$ on both sides of \eqref{gen-eq}, and substituting \eqref{temp2-1}, one obtains the following equation satisfied by $Z_\qa(\ql)$ after a lengthy calculation:
\begin{equation*}
(g^{\qa\qb}-\ql\eta^{\qa\qb})Z_\qb(\ql)^\prime+\Qg_\qg^{\qa\qb}v^{\qg,1}Z_\qb(\ql) =0.
\end{equation*}
From the coefficient of $\ql^0$ and the non-degeneracy of $\eta^{\qa\qb}$ it follows that $z_{\qb,0}^\prime = 0$. Since we do not have odd constants in the ring $\mathcal{\hat A}(M)$ except for zero, we must have $z_{\qb,0} = 0$. Now we arrive at the vanishing of $z_{\qb,k}\,(k\ge 1)$ from the coefficients of $\ql^{-k}$ by induction. The lemma is proved.
\end{proof}

\begin{Th}[Super extension of the principal hierarchy]
\label{super-1}
We have the following mutually commuting flows associated with any given Frobenius manifold $M$:
\begin{align*}
&\diff{v^\qa}{t^{\qb,p}} = \eta^{\qa\qg}{(\qp_\ql\qp_\qg h_{\qb,p+1})v^{\ql,1}},\quad \diff{\qs_{\qa,k}}{t^{\qb,p}} = \eta^{\qg\qe}(\qp_\qa\qp_\qe h_{\qb,p+1})\qs_{\qg,k}^1,\\
&\diff{v^\qa}{\qt_m} = \eta^{\qa\qb}\qs_{\qb,m}^1,\quad
\diff{\qs_{\qa,k}}{\qt_m} = -\diff{\qs_{\qa,m}}{\qt_k} = \Qg^{\qg\qb}_\qa\sum_{i = 0}^{m-k-1}\qs_{\qb,k+i}\qs_{\qg,m-i-1}^1,\quad 0\le k\leq m.
\end{align*}
where $\qa,\qb=1,\dots,n$, and $m, p\ge 0$.
\end{Th}
\begin{proof}
The following identities follow from Lemma \ref{zh-9} and a direct computation:
\begin{equation}
\label{temp2-2}
\left[\diff{}{t^{\qa,p}},\diff{}{t^{\qb,q}}\right]v^\qg = \left[\diff{}{t^{\qa,p}},\diff{}{\qt_k}\right]v^\qg = \left[\diff{}{\qt_m},\diff{}{\qt_k}\right]v^\qg = 0.
\end{equation}

By using Lemma \ref{zh-9} and the fact that the right and side of \eqref{ph} do not depend on the super variables $\qs_{\qa, l}$ we arrive at
the relation
\begin{equation*}
\diff{}{t^{\qa,p}}\diff{}{t^{\qb,q}}\diff{v^\qg}{\qt_k} = \diff{}{t^{\qb,q}}\diff{}{t^{\qa,p}}\diff{v^\qg}{\qt_k}.
\end{equation*}
Thus from \eqref{v-tau} we obtain 
\[\left[\diff{}{t^{\qa,p}},\diff{}{t^{\qb,q}}\right]\qs_{\qg,k}=0.\]

Let us proceed to prove the validity of the commutation relation
\begin{equation}\label{zh-10}
\left[\diff{}{t^{\qa,p}},\diff{}{\qt_m}\right]\qs_{\qb,k}=0.
\end{equation}
It is easy to show the above commutation relation for $m=k=0$. By using the recursion relation \eqref{biham-rec} and induction on $k$ we know that the commutation relation \eqref{zh-10} also holds true
for $m=0$ and $k\ge 1$. To prove the general case, we first note the validity of the following identity due to Lemma \ref{zh-9}:
\begin{equation*}
\diff{\qs_{\qb,k}}{t^{\qa,p}} = \diff{}{\qt_k}(\qp_\qb h_{\qa,p+1}),
\end{equation*}
from this and the relation given in \eqref{temp2-2} it follows that
\begin{equation*}
\diff{}{\qt_m}\diff{\qs_{\qb,0}}{t^{\qa,p}} = \diff{}{\qt_m}\diff{}{\qt_0}(\qp_\qb h_{\qa,p+1}) = -\diff{}{\qt_0}\diff{}{\qt_m}(\qp_\qb h_{\qa,p+1}) = -\diff{}{\qt_0}\diff{\qs_{\qb,m}}{t^{\qa,p}}.
\end{equation*}
By using \eqref{temp2-2} again we obtain
\begin{align*}
&\eta_{\qe\qb}\left(\diff{}{t^{\qa,p}}\diff{\qs_{\qb,0}}{\qt_m}-\diff{}{\qt_m}\diff{\qs_{\qb,0}}{t^{\qa,p}}\right)^\prime = \eta_{\qe\qb}\left(\diff{}{t^{\qa,p}}\diff{\qs_{\qb,0}}{\qt_m}+\diff{}{\qt_0}\diff{\qs_{\qb,m}}{t^{\qa,p}}\right)^\prime\\
&=\diff{}{t^{\qa,p}}\diff{}{\qt_m}\diff{v^\qe}{\qt_0}+\diff{}{\qt_0}\diff{}{t^{\qa,p}}\diff{v^\qe}{\qt_m} = \diff{}{\qt_0}\diff{}{t^{\qa,p}}\diff{v^\qe}{\qt_m}-\diff{}{t^{\qa,p}}\diff{}{\qt_0}\diff{v^\qe}{\qt_m}\\
&=-\eta_{\qe\qb}\left(\diff{}{t^{\qa,p}}\diff{\qs_{\qb,m}}{\qt_0}-\diff{}{\qt_0}\diff{\qs_{\qb,m}}{t^{\qa,p}}\right)^\prime.
\end{align*}
Thus we arrive at the validity of the commutation relation \eqref{zh-10}
for $m\ge 0$, $k=0$. Now by using the recursion relation \eqref{biham-rec} and induction on $k$ again we arrive at the validity of the commutation relation \eqref{zh-10} for general $m\ge 0$, $k\ge 0$.
The commutation relation $[\diff{}{\qt_m},\diff{}{\qt_l}]\qs_{\qa,k} = 0$
can be proved in a similar way.
The theorem is proved.
\end{proof}

\begin{Ex}
\label{kdv}
Consider the $1$-dimensional Frobenius manifold $M$. Its principal hierarchy is the Riemann hierarchy, which is also knows as the dispersionless KdV hierarchy
\begin{equation*}
\diff{v}{t_p} = R_{p+1}^\prime,\quad R_{p} = \frac{v^p}{p!},\quad p\ge 0.
\end{equation*}
Here we denote $v^1$ by $v$, and denote the time variables $t^{1,p}$ by $t_p$. The bihamiltonian structure reads:
\begin{equation*}
P_0 = \frac{1}{2}\int \qth\qth^1;\quad P_1 = \frac{1}{2}\int v\qth\qth^1.
\end{equation*}
The super extension of the principal hierarchy given in Theorem \ref{super-1} has the expression
\begin{align*}
&\diff{v}{t_p} = \frac{v^p}{p!} v',\quad \diff{\qs_k}{t_p} = \frac{v^p}{p!}\qs_k^\prime,\\
&\diff{v}{\qt_m} = \qs_m^\prime,\quad
\diff{\qs_k}{\tau_m} = -\diff{\qs_m}{\tau_k} = \frac{1}{2}\sum_{i=0}^{m-k-1}\qs_{i+k}\qs_{m-i-1}^\prime,\quad \text{for $0\le k\leq m$},
\end{align*}
where we denote $\qs_{1,k}$ by $\qs_k$.
\end{Ex}


\section{The super tau-cover of the principal hierarchy and its Virasoro symmetries}
\label{sec4}
In this section we will construct a tau-cover of the super extension of the principal hierarchy associated with an $n$-dimensional Frobenius manifold $M$ and consider its Virasoro symmetries. We will use the notations that are introduced in the last section. 

Let us first recall the tau-structure of the principal hierarchy.
Consider the functions $\Qo_{\qa,p;\qb,q}(v)$, $\qa,\qb = 1,\cdots,n$; $p,q\geq 0$ determined by the following generating function \cite{dubrovin2001normal}:
\begin{equation}
\label{omega}
\sum_{p\geq 0,q\geq 0}\Qo_{\qa,p;\qb,q}(v)z_1^pz_2^q = \frac{\langle\nabla h_\qa(v,z_1),\nabla h_\qb(v,z_2)\rangle - \eta_{\qa\qb}}{z_1+z_2}.
\end{equation}
Then we have the following identities:
\begin{equation}
\label{omg-ini}
\Qo_{\qa,p;1,0} = h_{\qa,p}(v),\quad \Qo_{\qa,p;\qb,0} = \qp_\qb h_{\qa,p+1}.
\end{equation}
By using these functions, we can define a function $\mathcal F_0(t)$ associated with any solution $v(t)=(v^1(t),\dots, v^n(t))$ of the principal hierarchy \eqref{ph}, such that
\begin{equation*}
\diff{}{t^{\qa,p}}\diff{\mathcal F_0}{t^{\qb,q}} = \Qo_{\qa,p;\qb,q}(v(t)).
\end{equation*}
The function $\mathcal F_0$ is called the genus zero free energy and its exponential $Z = \exp(\mathcal F_0)$ is called the tau function associated with a given solution of the principal hierarchy \eqref{ph}, and the functions $f_{\qa,p} = \diff{\mathcal F_0}{t^{\qa,p}}$ are called the one-point functions which satisfy the following tau-cover of the principal hierarchy:
\begin{align}
\diff{f_{\qa,p}}{t^{\qb,q}} &= \Qo_{\qa,p;\qb,q},\label{zh-11}\\
\diff{v^\qa}{t^{\qb,q}} &= \eta^{\qa\qe}(\qp_\qe\qp_\qg h_{\qb,q+1})v^{\qg,1}.\label{zh-11b}
\end{align}
We want to define a super extension of the principal hierarchy which contains the equations given in Theorem \ref{super-1} and the equations \eqref{zh-11}. To this end, we need to determine the flows $\diff{f_{\qa,p}}{\qt_n}$. By using the identities given in \eqref{omg-ini} we obtain
\begin{equation*}
\left(\diff{f_{\qa,p}}{\qt_n}\right)^\prime = \diff{}{\qt_n}\diff{f_{\qa,p}}{t^{1,0}} = \diff{h_{\qa,p}}{\qt_n}.
\end{equation*}
Hence we need to show that $\diff{h_{\qa,p}}{\qt_n}$ are the $x$-derivatives of some differential polynomials. Indeed, we have the following lemma.
\begin{Lem}
\label{phi}
Assume that $\frac{1-2k}{2}\notin\mathrm{Spec}(\mu)$ for any $k = 1,2,\cdots$. Then for any $p,n\geq 0$ there exists $\Phi_{\qa,p}^n \in\mathcal{\hat A}(M)$ such that
\begin{equation*}
\diff{h_{\qa,p}}{\qt_n} = (\Phi_{\qa,p}^n)^\prime.
\end{equation*}
\end{Lem}
\begin{proof}
We prove the lemma by induction on $p$. For $p = 0$ this is obvious, since by using \eqref{hamil-ini} and \eqref{v-tau} we have
\begin{equation*}
\diff{h_{\qa,0}}{\qt_n} = \eta_{\qa\qg}\diff{v^\qg}{\qt_n} = \qs_{\qa,n}^\prime.
\end{equation*}
For $p\geq 1$, by using the relation \eqref{hamil-rec} we obtain
\begin{equation*}
\diff{h_{\qa,p}}{\qt_n} = -c^{\ql\qe}_\qz(\qp_\ql h_{\qa,p-1})v^{\qz,1}\qs_{\qe,n}+(\eta^{\ql\qe}(\qp_\ql h_{\qa,p})\qs_{\qe,n})^\prime.
\end{equation*}
From this equation and the relations \eqref{hamil-rec}, \eqref{biham-rec}, \eqref{gam-v} and \eqref{v-tau} it follows that
\begin{align*}
\diff{h_{\qa,p}}{\qt_n}= &-2\mu_\qe c^{\ql\qe}_\qz(\qp_\ql h_{\qa,p-1})v^{\qz,1}\qs_{\qe,n}+2\eta^{\qb\qe}E(\qp_\qb h_{\qa,p})\qs_{\qe,n}^1\\&-2\diff{h_{\qa,p-1}}{\qt_{n+1}}+(\eta^{\ql\qe}(\qp_\ql h_{\qa,p})\qs_{\qe,n})^\prime.
\end{align*}
By using the quasi-homogeneous condition \eqref{homog} and the relations 
\eqref{mu-eta}, \eqref{hamil-rec} we arrive at the equations
\begin{align*}
\diff{h_{\qa,p}}{\qt_n}= &-2\mu_\qe c^{\ql\qe}_\qz(\qp_\ql h_{\qa,p-1})v^{\qz,1}\qs_{\qe,n}+2\mu_\qb\eta^{\qb\qe}(\qp_\qb h_{\qa,p})\qs_{\qe,n}^1-2\diff{h_{\qa,p-1}}{\qt_{n+1}}\\&+(\eta^{\ql\qe}(\qp_\ql h_{\qa,p})\qs_{\qe,n})^\prime+2(p+\mu_\qa)\diff{h_{\qa,p}}{\qt_n}+2\sum_{k=1}^p(R_k)^\xi_\qa\diff{h_{\xi,p-k}}{\qt_n}
\\
=&(2\mu_\qb\eta^{\qb\qe}(\qp_\qb h_{\qa,p})\qs_{\qe,n})^\prime
-2\diff{h_{\qa,p-1}}{\qt_{n+1}}+(\eta^{\ql\qe}(\qp_\ql h_{\qa,p})\qs_{\qe,n})^\prime\\&+2(p+\mu_\qa)\diff{h_{\qa,p}}{\qt_n}+2\sum_{k=1}^p(R_k)^\xi_\qa\diff{h_{\xi,p-k}}{\qt_n}.
\end{align*}
Thus, under the assumption $\frac{1-2k}{2}\notin\mathrm{Spec}(\mu)$ for any $k = 1,2,\cdots$, we can obtain $\Phi_{\qa,p}^n$ from the following recursion relation:
\begin{equation}\label{phi-rec}
-\left(\frac{2p-1}{2}+\mu_\qa\right)\Phi_{\qa,p}^n = \left(\frac 12+\mu_\ql\right)\eta^{\ql\qe}(\qp_\ql h_{\qa,p})\qs_{\qe,n}+\sum_{k=1}^p(R_k)^\xi_\qa\Phi_{\xi,p-k}^n-\Phi_{\qa,p-1}^{n+1}
\end{equation}
with the initial condition
\begin{equation*}
\Phi_{\qa,0}^n = \qs_{\qa,n}.
\end{equation*}
The lemma is proved.
\end{proof}

If the condition $\frac{1-2k}{2}\notin\mathrm{Spec}(\mu)$ for any $k = 1,2,\cdots$ is satisfied, we call the Frobenius manifold $M$ non-resonant; otherwise we call $M$ resonant. Note that 
the notion of resonance and non-resonance used here are different from the one used in \cite{dubrovin1996geometry}. 
For the non-resonant case, we see from Lemma \ref{phi} that all the flows $\diff{f_{\qa,p}}{\qt_n}$ are well-defined. However if $M$ is resonant, then the recursion relation \eqref{phi-rec} becomes trivial for $p, \qa$ with $1-2p-2\mu_\qa=0$. In this case, we need to regard $\Phi^n_{\qa,p}$ as new super variables satisfying the following relation:
\begin{equation}
\label{phi-def}
(\Phi_{\qa,p}^n)^\prime = \diff{h_{\qa,p}}{\qt_n} = \qp_\qb h_{\qa,p}\eta^{\qb\qg}\qs_{\qg,n}^1.
\end{equation}
and these new super variables are new unknown functions whose evolution should be included in the super tau cover of the principal hierarchy. 

In order to derive the flows with respect to the variables $\Phi_{\qa,p}^n$, in view of the relation \eqref{phi-def} we are to find differential polynomials whose derivative with respect to the spatial variable $x$ yield $\frac{\qp^2 h_{\qa,p}}{\qp t^{\qb,q}\qp\qt_n}$ and $\frac{\qp^2 h_{\qa,p}}{\qp\qt_k\qp\qt_n}$. Thus let us introduce
\begin{equation}
\label{del-def}
\Qd_{\qa,p}^{k,n}=-\Qd_{\qa,p}^{n,k}=\eta^{\qg\ql}\qp_\ql h_{\qa,p}\Qg_\qg^{\qd\mu}\left(\sum_{i=0}^{k-n-1}\qs_{\mu,n+i}\qs_{\qd,k-i-1}^1\right),\quad k\ge n,
\end{equation}
By using these notations, we have the following lemma which holds true for any Frobenius manifold (whether resonant or not):
\begin{Lem}
\label{exp-phi}
The differential polynomials $\frac{\qp^2 h_{\qa,p}}{\qp t^{\qb,q}\qp\qt_n}$ and $\frac{\qp^2 h_{\qa,p}}{\qp\qt_k\qp\qt_n}$ belong to
$\qp_x\hat{\mathcal{A}}(M)$. More explicitly, we have:
\begin{equation*}
\frac{\qp^2 h_{\qa,p}}{\qp t^{\qb,q}\qp\qt_n} = \qp_x\diff{\Qo_{\qa,p;\qb,q}}{\qt_n},\quad \frac{\qp^2 h_{\qa,p}}{\qp\qt_k\qp\qt_n} = \qp_x \Qd_{\qa,p}^{k,n}.
\end{equation*}
\end{Lem}
\begin{proof}
By using the definitions \eqref{ph}, \eqref{omega} of the principal hierarchy and the functions $\Omega_{\qa,p;\beta,q}$, it is straightforward to check the relation
\begin{equation}\label{zh-12-3-1}
\frac{\qp^2 h_{\qa,p}}{\qp t^{\qb,q}\qp\qt_n} = \qp_x\diff{\Qo_{\qa,p;\qb,q}}{\qt_n}\in \qp_x\hat{\mathcal{A}}(M).
\end{equation} 

Let us proceed to prove the other relations. Without loss of generality, we assume that $k\geq n$ and we shall prove the identity\begin{equation*}
\frac{\qp^2 h_{\qa,p}}{\qp\qt_k\qp\qt_n} = \qp_x \Qd_{\qa,p}^{k,n},\quad k\geq n,
\end{equation*}
by induction on $k-n$. The case $k-n=0$ is trivial, so we first consider $k=n+1$. In this case we have
\begin{align*}
&\diff{}{\qt_{n+1}}(\eta^{\ql\qg}\qp_\ql h_{\qa,p}\qs_{\qg,n}^1)\\
=&\eta^{\ql\qg}\qp_\qb\qp_\ql h_{\qa,p}\eta^{\qb\qe}\qs_{\qe,n+1}^1\qs_{\qg,n}^1+\eta^{\ql\qg}\qp_\ql h_{\qa,p}\left(\Qg_\qg^{\qd\mu}\qs_{\mu,n}\qs_{\qd,n}^1\right)^\prime\\
=&\eta^{\ql\qg}c_{\qb\ql}^\qd\qp_\qd h_{\qa,p-1}(g^{\qb\qe}\qs_{\qe,n}^1+\Qg_\qz^{\qb\qe}v^{\qz,1}\qs_{\qe,n})\qs_{\qg,n}^1-\eta^{\ql\qd}c_{\mu\ql}^\qz\qp_\qz h_{\qa,p-1}v^{\mu,1}\Qg_\qd^{\qg\qe}\qs_{\qe,n}\qs_{\qg,n}^1\\
&+\left(\eta^{\ql\qg}\qp_\ql h_{\qa,p}\Qg_\qg^{\qd\mu}\qs_{\mu,n}\qs_{\qd,n}^1\right)^\prime,
\end{align*}
here we use the recursion relation \eqref{hamil-rec} and the bihamiltonian recursion relation\eqref{biham-rec}.
From the associativity equation of the structure constants $c_{\qa}^{\qb\qg}$ and the identity \eqref{gam-v}, it is easy to see that 
\begin{equation*}
\eta^{\ql\qg}c_{\qb\ql}^\qd\qp_\qd h_{\qa,p-1}\Qg_\qz^{\qb\qe}v^{\qz,1}-\eta^{\ql\qd}c_{\mu\ql}^\qz\qp_\qz h_{\qa,p-1}v^{\mu,1}\Qg_\qd^{\qg\qe}=0.
\end{equation*}
Thus we arrive at the following equation:
\begin{equation*}
\diff{}{\qt_{n+1}}\diff{h_{\qa,p}}{\qt_n} = \eta^{\ql\qg}c_{\qb\ql}^\qd\qp_\qd h_{\qa,p-1}g^{\qb\qe}\qs_{\qe,n}^1\qs_{\qg,n}^1+\left(\eta^{\ql\qg}\qp_\ql h_{\qa,p}\Qg_\qg^{\qd\mu}\qs_{\mu,n}\qs_{\qd,n}^1\right)^\prime.
\end{equation*}
By using the identity \eqref{gam-v} and the associativity equation again, we know that the expressions
$\eta^{\ql\qg}c_{\qb\ql}^\qd g^{\qb\qe}$
are symmetric with respect to the indices $\qg$ and $\qe$, hence the term $\eta^{\ql\qg}c_{\qb\ql}^\qd\qp_\qd h_{\qa,p-1}g^{\qb\qe}\qs_{\qe,n}^1\qs_{\qg,n}^1$ vanishes in the r.h.s. of the above equation. This completes the proof of the case $k-n=1$.

Assume that the relation \eqref{zh-12-3-1} holds true for $k-n\le l$ with $l\ge 1$. When $k=n+l+1$ we have
\begin{align*}
&\diff{}{\qt_{n+l+1}}(\eta^{\ql\qg}\qp_\ql h_{\qa,p}\qs_{\qg,n}^1)\\
=&\diff{}{\qt_{n+l+1}}(\eta^{\ql\qg}\qp_\ql h_{\qa,p})\qs_{\qg,n}^1+\eta^{\ql\qg}\qp_\ql h_{\qa,p}\left(\Qg_\qg^{\qd\mu}\sum_{i=0}^l\qs_{\mu,n+i}\qs_{\qd,n+l-i}^1\right)^\prime\\
=&\diff{}{\qt_{n+l+1}}(\eta^{\ql\qg}\qp_\ql h_{\qa,p})\qs_{\qg,n}^1+\eta^{\ql\qg}\qp_\ql h_{\qa,p}\left(\Qg_\qg^{\qd\mu}\sum_{i=0}^{l-2}\qs_{\mu,n+1+i}\qs_{\qd,n+l-1-i}^1\right)^\prime\\
&+\eta^{\ql\qg}\qp_\ql h_{\qa,p}\left(\Qg_\qg^{\qd\mu}(\qs_{\mu,n}\qs_{\qd,n+l}^1+\qs_{\mu,n+l}\qs_{\qd,n}^1)\right)^\prime\\
=&\diff{}{\qt_{n+l+1}}(\eta^{\ql\qg}\qp_\ql h_{\qa,p})\qs_{\qg,n}^1+\diff{}{\qt_{n+l}}\diff{h_{\qa,p}}{\qt_{n+1}}-\diff{}{\qt_{n+l}}(\eta^{\ql\qg}\qp_\ql h_{\qa,p})\qs_{\qg,n+1}^1\\
&+\eta^{\ql\qg}\qp_\ql h_{\qa,p}\left(\Qg_\qg^{\qd\mu}(\qs_{\mu,n}\qs_{\qd,n+l}^1+\qs_{\mu,n+l}\qs_{\qd,n}^1)\right)^\prime.
\end{align*}
Then we can repeat exactly the same argument as we do in the case $k=n+1$ to obtain
\begin{equation*}
\diff{}{\qt_{n+l+1}}\diff{h_{\qa,p}}{\qt_n} = \diff{}{\qt_{n+l}}\diff{h_{\qa,p}}{\qt_{n+1}}+\left(\eta^{\ql\qg}\qp_\ql h_{\qa,p}\Qg_\qg^{\qd\mu}(\qs_{\mu,n}\qs_{\qd,n+l}^1+\qs_{\mu,n+l}\qs_{\qd,n}^1)\right)^\prime.
\end{equation*}
So by applying the induction hypothesis we arrive at
\[
\diff{}{\qt_{n+l+1}}\diff{h_{\qa,p}}{\qt_n} =\qp_x\Qd_{\qa,p}^{n+l+1,n} \in \qp_x\hat{\mathcal{A}}(M).
\]
The Lemma is proved.
\end{proof}

\begin{Ex}
Let $M$ be the Frobenius manifold given by the quantum cohomology of $\mathbb{CP}^1$. Its Frobenius structure is characterized by the potential
\begin{equation*}
F = \frac{1}{2}v^2u+e^u.
\end{equation*}
Note that in this example $v, u$ are used to denote the flat coordinates and $\qp_v$ is the unit vector field. The potential $F$ is quasi-homogeneous with respect to the following Euler vector field:
\begin{equation*}
E = v\diff{}{v}+2\diff{}{u}.
\end{equation*}
We have the following monodromy data of the Frobenius manifold: 
\begin{equation*}
\mu=\begin{pmatrix}-\frac{1}{2}& 0\\ 0&\frac{1}{2}\end{pmatrix},\quad R=R_1=\begin{pmatrix}0&0\\2&0\end{pmatrix}.
\end{equation*}
So this Frobenius manifold is resonant in the above sense. 
The equations \eqref{hamil-rec}--\eqref{homog} for the functions
$h_{\qa,p}$ can be represented in the form
\begin{align*}
&\qp_v h_{\qa,p+1}=h_{\qa,p},\quad\qp_u\qp_u h_{\qa,p+1}=e^u\qp_v h_{\qa,p},\\
&\qp_Eh_{\qa,p}=\left(p+\frac{1}{2}+\mu_\qa\right)h_{\qa,p}+2\qd_{\qa,1}h_{2,p-1}.
\end{align*}
The first few $h_{\qa,p}$ have the expressions
\begin{align*}
&h_{1,0}=u,\quad h_{2,0}=v,\\
&h_{1,1}=uv,\quad h_{2,1}=\frac{1}{2}v^2+e^u,\\
&h_{1,2}=\frac{1}{2}v^2u+ue^u-2e^u,\quad h_{2,2}=\frac{1}{6}v^3+ve^u.
\end{align*}
They yield the following flows of the principal hierarchy:
\begin{align*}
&\diff{v}{t^{1,0}}=v^\prime,\ \diff{u}{t^{1,0}}=u^\prime;\quad \diff{v}{t^{2,0}}=e^uu^\prime,\diff{u}{t^{2,0}}=v^\prime,\\
&\diff{v}{t^{1,1}} = vv^\prime+e^uu^\prime,\ \diff{u}{t^{1,1}} = u^\prime v+uv^\prime,\\
&\diff{v}{t^{2,1}}=e^uv^\prime+ve^uu^\prime,\ \diff{u}{t^{2,1}}=vv^\prime+e^uu^\prime.
\end{align*}
Let us extend the principal hierarchy to include the odd flows. Since the principal hierarchy is bihamiltoanin with the Hamiltonian operators
\begin{equation*}
\mathcal{P}_0 = \begin{pmatrix}0 &\qp_x\\ \qp_x &0 \end{pmatrix},\quad \mathcal{P}_1=\begin{pmatrix}e^uu^\prime+2e^u\qp_x&v\qp_x\\v^\prime+v\qp_x&2\qp_x\end{pmatrix},
\end{equation*}
the super variables $\qs_{\qa,n}$ satisfy the bihamiltonian recursion
relation
\begin{equation*}
\qs_{1,n+1}=v\qs_{1,n}+2\qs_{2,n},\quad \qs_{2,n+1}^1=2e^u\qs_{1,n}^1+v\qs_{2,n}^1+e^uu^\prime\qs_{1,n}.
\end{equation*} 
Then the odd flows of the principal hierarchy are given by
\begin{align*}
&\diff{v}{\qt_n}=\qs_{2,n}^1,\quad \diff{u}{\qt_n}=\qs_{1,n}^1,\\
&\diff{\qs_{1,n}}{t^{\qa,p}}=(\qp_v\qp_uh_{\qa,p+1})\qs_{1,n}^1+(\qp_v\qp_vh_{\qa,p+1})\qs_{2,n}^1,\\
&\diff{\qs_{2,n}}{t^{\qa,p}}=(\qp_u\qp_uh_{\qa,p+1})\qs_{1,n}^1+(\qp_v\qp_uh_{\qa,p+1})\qs_{2,n}^1,\\
&\diff{\qs_{1,n}}{\qt_k}=-\diff{\qs_{1,k}}{\qt_n}=\sum_{i=0}^{k-n-1}\qs_{1,n+i}\qs_{2,k-i-1}^1,\quad n\leq k,\\&\diff{\qs_{2,n}}{\qt_k}=-\diff{\qs_{2,k}}{\qt_n}=e^u\sum_{i=0}^{k-n-1}\qs_{1,n+i}\qs_{1,k-i-1}^1,\quad n\leq k.
\end{align*}
Now let us proceed to write down its super tau-cover. Introduce the one-point functions $f_{\qa,p}$, then we need to determine the odd flows $\diff{f_{\qa,p}}{\qt_n}$. Due to Lemma \ref{phi}, we know that the super variables $\Phi_{\qa,p}^n$ satisfy the recursion relations
\begin{align*}
(p-1)\Phi_{1,p}^n&=\Phi_{1,p-1}^{n+1}-\qp_uh_{1,p}\qs_{1,n}-2\Phi_{2,p-1}^n,\\
p\Phi_{2,p}^n&=\Phi_{2,p-1}^{n+1}-\qp_uh_{2,p}\qs_{1,n}
\end{align*}
with initial condition $\Phi_{\qa,0}^n=\qs_{\qa,n}$, from which it follows that all $\Phi_{2,p}^n$ are differential polynomials of $u,v,\qs_{\qa,k}$, and $\Phi_{1,p}^n$ are differential polynomials of $u,v,\qs_{\qa,k},\Phi_{1,1}^k$. So in this example we only need to introduce the super variables $\Phi_{1,1}^n$ and they satisfy the relation:
\begin{equation*}
(\Phi_{1,1}^n)^\prime = v\qs_{1,n}^1+u\qs_{2,n}^1.
\end{equation*}
By using Lemma \ref{exp-phi}, we can write down the evolution of $\Phi_{1,1}^n$ as follows:
\begin{align*}
\diff{\Phi_{1,1}^n}{t^{\qa,p}}=&(\qp_v\qp_uh_{\qa,p+1}\qs_{1,n}^1+\qp_v\qp_vh_{\qa,p+1}\qs_{2,n}^1)v\\
&+(\qp_u\qp_uh_{\qa,p+1}\qs_{1,n}^1+\qp_v\qp_uh_{\qa,p+1}\qs_{2,n}^1)u,\\
\diff{\Phi_{1,1}^n}{\qt_k}=&-\diff{\Phi_{1,1}^k}{\qt_n}=\sum_{i=0}^{k-n-1}v\qs_{1,n+i}\qs_{2,k-i-1}^1+ue^u\qs_{1,n+i}\qs_{1,k-i-1}^1,\ n\leq k.
\end{align*}
\end{Ex}

To simplify the notation, in what follows we do not distinguish the resonant case and the non-resonant case. We will always include the variables $\Phi_{\qa,p}^n$ and their evolution as part of the super tau cover of the principal hierarchy. One note that for non-resonant case these variables are redundant due to Lemma \ref{phi}.
We summarize our result on the super tau-cover of the principal hierarchy in the following theorem.
\begin{Th}[Super tau-cover of the principal hierarchy]
\label{super-2}
Let $M$ be an $n$-dimensional Frobenius manifold. Then the following equations, together with the ones described in Theorem \ref{super-1}, give the super tau-cover of the principal hierarchy associated with $M$:
\begin{align}
\label{f-t}
\diff{f_{\qa,p}}{t^{\qb,q}} &= \Qo_{\qa,p;\qb,q},\\
\label{f-tau}
\diff{f_{\qa,p}}{\qt_n} &= \Phi_{\qa,p}^n,\\
\label{phi-t}
\diff{\Phi_{\qa,p}^n}{t^{\qb,q}} &= \diff{\Qo_{\qa,p;\qb,q}}{\qt_n},\\
\label{phi-tau}
\diff{\Phi_{\qa,p}^n}{\qt_k} &= \Qd_{\qa,p}^{k,n},
\end{align}
where $\Qo_{\qa,p;\qb,q}$ are defined in \eqref{omega} and $\Qd_{\qa,p}^{k,n}$ are defined via \eqref{del-def}.
\end{Th}
\begin{proof}
We only need to prove that these flows commute. This can easily be checked by using the relation $f_{\qa,p}^\prime = h_{\qa,p}$ and \eqref{phi-def}, and by applying Theorem \ref{super-1}.
\end{proof}

The tau-cover is introduced to study the Virasoro symmetry of the principal hierarchy. For any given Frobenius manifold $M$, a representation of (half of) the Virasoro algebra is constructed in \cite{dubrovin1999frobenius} via a set of linear differential operators, which also gives a representation of the additional symmetries of the principal hierarchy. Let us recall the basic construction very briefly. We first construct a Heisenberg algebra by introducing the following operators:
\begin{equation*}
a_k^\qa = \begin{cases}\eta^{\qa\qb}\diff{}{t^{\qb,k}},\quad &k\geq 0,\\(-1)^{k+1}t^{\qa,-k-1},\quad &k<0,\end{cases}
\end{equation*}
they satisfy the commutation relations
\begin{equation*}
[a_k^\qa,a_l^\qb] = (-1)^k\eta^{\qa\qb}\qd_{k+l+1,0}.
\end{equation*}

Define the following matrices for $m\geq -1$:
\begin{equation*}
P_m(\mu,R) = \begin{cases}[\exp(R\partial_x)\Pi_{j=0}^m(x+\mu+j-\frac{1}{2})]_{x=0},\quad &m \geq 0\\ 1, \quad & m = -1.\end{cases}
\end{equation*}
Let $V$ be an $n$-dimensional complex vector space with a fixed basis $\{e_\qa\}$, and endowed with a symmetric bilinear form $\langle-,-\rangle$ defined by $\eta_{\qa\qb} = \langle e_\qa,e_\qb\rangle$. The matrices $P_m$ induces endomorphisms of $V$ via the given basis. Introduce the vector-valued operators $a_k = a_k^\qa e_\qa$, then we can define the following operators of Sugawara-type for $m\geq -1$:
\begin{equation}
\label{L0}
L_m^{even}  =\frac{1}{2}\sum_{k,l\in\mathbb Z}(-1)^{k+1}:\langle a_l,[P_m(\mu-k, R)]_{m-1-l-k}a_k\rangle:+\frac{1}{4}\qd_{m,0}\mathrm{tr}\left(\frac{1}{4}-\mu^2\right),
\end{equation}
here the normal ordering is given by putting the annihilators (i.e., $a^\qa_k$ for $k\geq 0$) on the right, and the matrix $[P_m(\mu-k, R)]_{m-1-l-k}$ is uniquely determined from $P_m(\mu-k,R)$ in a certain way, see \cite{dubrovin1999frobenius} for its precise definition. These linear differential operators can be rewritten in the forms
\[L^{even}_m=a_m^{\qa,p;\qb,q}\frac{\qp^2}{\qp t^{\qa,p}\qp t^{\qb,q}}+{b_m}_{\qa,p}^{\qb,q} t^{\qa,p}\frac{\qp}{\qp t^{\qb,q}}+
c^m_{\qa,p;\qb,q} t^{\qa,p} t^{\qb,q}+\frac{1}{4}\qd_{m,0}\mathrm{tr}\left(\frac{1}{4}-\mu^2\right).
\]
We introduce the following infinitesimal transformation of the genus zero free enery $\mathcal{F}_0$:
\begin{align}
\diff{\mathcal{F}_0}{s_m}&=a_m^{\qa,p;\qb,q}f_{\qa,p}f_{\qb,q}+
b_{m;\qa,p}^{\qb,q} t^{\qa,p}f_{\qb,q}+
c_{m;\qa,p;\qb,q} t^{\qa,p} t^{\qb,q}\\
&=\textrm{Coeff}\left(e^{-\qe^{-2}\mathcal{F}_0}\tilde{L}_m^{even}e^{\qe^{-2}\mathcal{F}_0}, \qe^{-2}\right),
\end{align}
where 
\[\tilde{L}_m^{even}=\left.L_m^{even}\right|_{t^{\qa,p}\mapsto\qe^{-1} t^{\qa,p}, \frac{\qp}{\qp t^{\qa,p}}\mapsto \qe \frac{\qp}{\qp t^{\qa,p}}}.\]
Then the Virasoro symmetries of the principal hierarchy is given by
\begin{equation}
\label{temp3-1}
\diff{f_{\qa,p}}{s_m} := \diff{}{t^{\qa,p}}\diff{\mathcal F_0}{s_m};\quad \diff{v^\qa}{s_m} := \eta^{\qa\qb}\frac{\qp^2}{\qp t^{\qb,0}\qp t^{1,0}}\diff{\mathcal F_0}{s_m},\quad m\ge -1.
\end{equation}
It was proved that
\begin{Th}[\cite{dubrovin1999frobenius}]\mbox{}
\begin{enumerate}
\item The operators $L_m^{even}$ satisfy the Virasoro commutation relations 
\begin{equation}\label{zh-12}
[L_k^{even},L_m^{even}] = (k-m)L_{k+m}^{even},\quad k, m\ge -1.
\end{equation}
	\item The flows $\diff{}{s_m}$ defined in \eqref{temp3-1} are symmetries of the tau-cover of the principal hierarchy \eqref{zh-11}, \eqref{zh-11b}, meaning that
	\[
	\left[\diff{}{s_m},\diff{}{t^{\qa,p}}\right] = 0.
	\]
\end{enumerate}
\end{Th}
To construct the Virasoro symmetries of the super tau-cover of the principal hierarchy, we need to modify the operators $L^{even}_m$ to include the odd time variables $\qt_n$.
\begin{Lem}
\label{Vir}
Let $c_0\in\mathbb C$ be an arbitrary constant, and denote:
\begin{equation}
\label{L}
L_m = L_m^{even}+L_m^{odd},\quad L^{odd}_m = \sum_{k\geq 0}(k+c_0)\qt_k\diff{}{\qt_{k+m}},\quad m\geq -1.
\end{equation}
Then the operators $L_m$ satisfy the commutation relations $[L_m,L_n] = (m-n)L_{m+n}$.
\end{Lem}
\begin{proof}
It is easy to show that 
\[[L_m^{odd},L_n^{odd}] = (m-n)L_{m+n}^{odd}.\]
Then the lemma follows from \eqref{zh-12} and the fact that $[L_m^{even},L_n^{odd}] = 0$.
\end{proof}
\begin{Th}
\label{super-3}
The super tau-cover of the principal hierarchy described in Theorem \ref{super-2} has the following Virasoro symmetries:
\begin{align*}
\frac{\qp f_{\qa, p}}{\qp s_m}&=
\frac{\qp f_{\qa, p}}{\qp s^{even}_m}+\frac{\qp f_{\qa, p}}{\qp s^{odd}_m},
\quad \diff{\Phi^n_{\qa,p}}{s_m}=\frac{\qp}{\qp \tau_n}\left(\frac{\qp f_{\qa,p}}{\qp s_m}\right),\\
\frac{\qp v^\qa}{\qp s_m}&=
\frac{\qp v^\qa}{\qp s^{even}_m}+\frac{\qp v^\qa}{\qp s^{odd}_m},\quad \frac{\qp \sigma_{\qa,p}}{\qp s_m}=\frac{\qp}{\qp \tau_p}\left(\frac{\qp f_{\qa,0}}{\qp s_m}\right),
\end{align*}
where $\frac{\qp}{\qp s^{even}_m}$, $\frac{\qp}{\qp s^{odd}_m}$  are defined respectively by the right hand sides of \eqref{temp3-1} 
and 
\[\frac{\qp f_{\qa, p}}{\qp s^{odd}_m}=\sum_{k\geq 0}(k+c_0)\qt_k\diff{f_{\qa,p}}{\qt_{k+m}},\quad 
\frac{\qp v^\qa}{\qp s^{odd}_m}=\eta^{\qa\qb}\sum_{k\geq 0}(k+c_0)\qt_k\frac{\qp}{\qp t^{1,0}}\diff{f_{\qb,0}}{\qt_{k+m}}\]
with the flows of $f_{\qa, p}$ along $\tau_k$ given by \eqref{f-tau}.
These Virasoro symmetries satisfy the Virasoro commutation relations
\[\left[\frac{\qp}{\qp s_k},\frac{\qp}{\qp s_m}\right]=(m-k) \frac{\qp}{\qp s_{k+m}},\quad k, m\ge -1.\]
\end{Th}
\begin{proof}
The relations
\[\left[\diff{}{s_m},\diff{}{t^{\qa,p}}\right]f_{\qb,q}=0,\quad \left[\diff{}{s_m},\diff{}{t^{\qa,p}}\right]v^\qb=0,\]
follow from the well-known facts proved in \cite{dubrovin1999frobenius}:
\[\left[\diff{}{s_m^{even}},\diff{}{t^{\qa,p}}\right]f_{\qb,q}=0,\quad \left[\diff{}{s_m^{even}},\diff{}{t^{\qa,p}}\right]v^\qb=0,\]
and the obvious relation
\[\left[\diff{}{s_m^{odd}},\diff{}{t^{\qa,p}}\right]=0.\]
At the same time the relations
\[\left[\diff{}{s_m},\diff{}{\qt_n}\right]f_{\qb,q}=0,\quad \left[\diff{}{s_m},\diff{}{\qt_n}\right]v^\qb=0,\]
are trivial due to the definitions.

Let us proceed to prove that
\begin{equation}
\label{Vir-pr1}
\left[\diff{}{s_m},\diff{}{t^{\qa,p}}\right]\Phi_{\qb,q}^n=0,\quad \left[\diff{}{s_m},\diff{}{t^{\qa,p}}\right]\qs_{\qb,n}=0.\end{equation}
In view of the fact that $\Phi_{\qb,0}^n =\qs_{\qb,n}$, clearly we only need to prove the first relation. We first recall a useful identity proved in \cite{dubrovin1999frobenius}.  Let us write the Virasoro operator as:
\begin{align*}
L_m &= \sum a_m^{\qa,p;\qb,q}\frac{\qp^2}{\qp t^{\qa,p}\qp t^{\qb,q}}+\mathcal L_m+\sum c^m_{\qa,p;\qb,q}t^{\qa,p}t^{\qb,q}+\frac{1}{4}\qd_{m,0}\mathrm{tr}\left(\frac{1}{4}-\mu^2\right),\\
\mathcal L_m &=\sum {b_m}_{\qb,q}^{\qa,p}t^{\qb,q}\diff{}{t^{\qa,p}}+\sum_{k\geq 0}(c_0+k)\qt_k\diff{}{\qt_{k+m}}.
\end{align*}
Then the coefficients appearing in the operator $L_m$ are related to the power of the Euler vector field $E$ by the following:
\begin{equation*}
E^{m+1}\Qo_{\qa,p;\qb,q} = 2 a_m^{\ql,k;\qe,l}\Qo_{\qa,p;\ql,k}\Qo_{\qe,l;\qb,q}+{b_m}^{\ql,k}_{\qa,p}\Qo_{\ql,k;\qb,q}+{b_m}^{\ql,k}_{\qb,q}\Qo_{\ql,k;\qa,p}+2c^m_{\qa,p;\qb,q}.
\end{equation*}
Using this relation, it is straightforward to conclude that proving the relation
\begin{equation*}
\left[\diff{}{s_m},\diff{}{t^{\qa,p}}\right]\Phi_{\qb,q}^n=0,
\end{equation*}
is equivalent to show that the following identity holds true:
\begin{equation*}
\diff{}{\qt_n}\left(E^{m+1}\Qo_{\qa,p;\qb,q}\right) = E^{m+1}\left(\qp^\qg\Qo_{\qa,p;\qb,q}\right)\qs_{\qg,n}^1+\qp^\qg\Qo_{\qa,p;\qb,q}\diff{}{\qt_n}\left(E^{m+1}v_{\qg}\right),
\end{equation*}
here $\qp^\qg = \eta^{\qg\ql}\qp_\ql$ and $v_\qg =  \eta_{\qg\ql}v^\ql$. This identity holds true trivially, indeed, for any vector field $X$ on the Frobenius manifold $M$, and any function $h$, it is easy to check that:
\[
\qs_{\qg,n}^1[\qp^\qg, X]h = \diff{}{\qt_n}(Xv_\qg)\qp^\qg h.
\]
Therefore we prove the validity of the relations \eqref{Vir-pr1}.

Finally we are to prove that
\begin{equation*}
\left[\diff{}{s_m},\diff{}{\qt_k}\right]\Phi_{\qb,q}^n=0,\quad \left[\diff{}{s_m},\diff{}{\qt_k}\right]\qs_{\qb,n}=0, \end{equation*}
which is equivalent to show that:
\begin{equation}
\label{Vir-pr2}
\diff{}{s_m}\diff{}{\qt_n}\diff{f_{\ql,0}}{\qt_k}-\diff{}{\qt_n}\diff{}{\qt_k}\diff{f_{\ql,0}}{s_m} = 0.
\end{equation}
Let us only prove \eqref{Vir-pr2} for the case $k=0, n=1$. The general case can be proved by induction using the definition \eqref{del-def}. By a straightforward computation, the validity of \eqref{Vir-pr2} is equivalent to prove the following two expressions are equal:
\begin{align}
\nonumber
A = &2 a_m^{\qa,p;\qb,q}\left(\diff{\Qo_{\qa,p;\ql,0}}{\qt_1}\Phi_{\qb,q}^0-\diff{\Qo_{\qa,p;\ql,0}}{\qt_0}\Phi_{\qb,q}^1\Qo_{\qa,p;\ql,0}\Qd^{1,0}_{\qb,q}\right)\\
\label{Vir-pr3}
&+{b_m}^{\qa,p}_{\ql,0}\Qd^{1,0}_{\qa,p}+\Qd_{\ql,0}^{m+1,0};\\
\nonumber
B = &E^{m+1}(\Qg^{\qd\mu}_\ql)\qs_{\mu,0}\qs_{\qd,0}^1+\Qg^{\qd\mu}_\ql\left({b_m}^{\qa,p}_{\mu,0}\Phi^0_{\qa,p}+2a_m^{\qa,p;\qb,q}\Qo_{\qa,p;\mu,0}\Phi^0_{\qb,q}\right)\qs_{\qd,0}^1\\
\nonumber
&+\Qg^{\qd\mu}_\ql\qs_{\mu,0}\left({b_m}^{\qa,p}_{\qd,0}\qp_x\Phi^0_{\qa,p}+{b_m}^{\qa,p}_{1,0}\diff{\qs_{\qd,0}}{t^{\qa,p}}\right)\\
\label{Vir-pr4}
&+2a_m^{\qa,p;\qb,q}\left(\Qo_{\qa,p;\qd,0}^\prime\Phi_{\qb,q}^0+\Qo_{\qa,p;\qd,0}\diff{\Qo_{\qb,q;1,0}}{\qt_0}+\diff{\Qo_{\qa,p;\qd,0}}{\qt_0}\Qo_{\qb,q;1,0}\right).
\end{align}

We prove that $A = B$ in the next Lemma and therefore the theorem is proved.

\end{proof}

\begin{Lem}
The two expressions \eqref{Vir-pr3} and \eqref{Vir-pr4} are equal.
\end{Lem}
\begin{proof}
We will still use $A$ to denote the expression \eqref{Vir-pr3} and $B$ for \eqref{Vir-pr4}. The proof is a straightforward computation for the case $-1\leq m\leq 2$.

For $m = -1$, according to \cite{dubrovin1999frobenius},the coefficients appearing in the Virasoro operator read: 
\[{b_{-1}}^{\qa,p}_{\qb,q} = \qd^{\qa}_\qb\qd^p_{q-1},\quad a_{-1}^{\qa,p;\qb,q} = 0.\]
So we easily see that $A = 0$ and $B = \qp_1\Qg^{\qd\mu}_\ql\qs_{\mu,0}\qs_{\qd,0}^1$. Using the identity \eqref{gam-v}, we conclude that $B = 0$.

For $m = 0$, the coefficients read \cite{dubrovin1999frobenius}:
\[{b_{0}}^{\qa,p}_{\qb,q} = \qd^{\qa}_\qb\qd^p_{q}\left(p+\frac 12+\mu_\qa\right)+\sum_{1\leq r\leq q}(R_r)^\qa_\qb\qd^p_{q-r},\quad a_{0}^{\qa,p;\qb,q} = 0.\]
Then we come to the expression for $A$ and $B$:
\begin{align*}
A &= \left(\frac{3}{2}+\mu_\ql\right)\diff{\qs_{\ql,0}}{\qt_1}\\
B &= E(\Qg_\ql^{\qa\qb})\qs_{\qb,0}\qs_{\qa,0}^1+\left(\frac{3}{2}+\mu_\qb+\mu_\qa+\mu_1\right)\Qg_\ql^{\qa\qb}\qs_{\qb,0}\qs_{\qa,0}^1
\end{align*}
Using the identity \eqref{gam-v} and the homogeneous condition \eqref{c-hom}, we derive the homogeneous condition satisfied by the Christoffel symbol:
\[E(\Qg_\ql^{\qa\qb}) = (\mu_\ql-\mu_\qa-\mu_\qb-\mu_1)\Qg_\ql^{\qa\qb}.\]
This proves $A=B$ for the case $m=0$.

For the case $m=1$, it is similar but more complicated. First we can compute the coefficients \cite{dubrovin1999frobenius}:
\begin{align*}
a_1^{\qa,p;\qb,q} &= \qd_p^0\qd_q^0\frac{1}{2}\eta^{\qa\qb}\left(\frac 12+\mu_\qa\right)\left(\frac 12+\mu_\qb\right);\\
{b_1}^{\qa,p}_{\qb,q} &= \qd^\qa_\qb\qd^p_{q+1}\left(q+\frac 12+\mu_\qa\right)\left(q+\frac 32+\mu_\qa\right)\\&+\sum_{r=1}^{q+1}(R_r)^\qa_\qb(2q+2+2\mu_\qb)\qd^p_{q-r+1}+\sum_{r=2}^{q+1}(R_{r,2})^\qa_\qb\qd^p_{q-r+1}.
\end{align*}
Then after considering the recursion relations \eqref{biham-rec}, we can divide the terms in $A$ and $B$ into three groups: the first group consists of the terms containing $\qs_{\zeta,1}\qs_{\qd,0}^1$, the second group consists of terms containing $\qs_{\zeta,0}\qs_{\qd,0}$ and the third consists of terms containing $\qs_{\zeta,0}\qs_{\qd,0}^1$. We need to verify that for each group, the terms in $A$ equal to those in $B$. The verification is straightforward and therefore we only take terms in the first group as an example. Computing the terms of the first group in $A$ and $B$, we conclude that the following identity needs verifying:
\begin{equation*}
\left(\frac 12+\mu_\zeta\right)\Qg^{\qd\zeta}_\ql\qs_{\zeta,1}\qs_{\qd,0}^1 = \eta^{\qa\qb}\left(\frac 12+\mu_\qa\right)\left(\frac 12+\mu_\qb\right)\qs_{\qb,1}\diff{\Qo_{\qa,0;\ql,0}}{\qt_0}
\end{equation*}
By applying the equation \eqref{hamil-rec} and \eqref{omg-ini}, one concludes that it is equivalent to prove the following:
\begin{equation*}
\Qg^{\qd\zeta}_\ql = \eta^{\qa\zeta}\left(\frac 12+\mu_\qa\right)\eta_{\qa\qb}c^{\qd\qb}_\ql.
\end{equation*}
This identity holds by using \eqref{gam-v} and \eqref{mu-eta}.

In a similar but much more involved way, one can show that for the case $m = 2$ the equality $A=B$ holds and the lemma is proved.

\end{proof}
\begin{Ex}
\label{Vir-kdv}
Let us continue the discussion given in Example \ref{kdv}. The super tau-cover of the Riemann hierarchy can be written down explicitly as described in Theorem \ref{super-2}. The flows of the one point functions
have the expressions
\begin{align*}
\diff{f_k}{t_n} &= \frac{u^{k+n+1}}{(n+k+1)n!k!},\\
\diff{f_k}{\qt_n} &= \frac{\Qg(\frac{1}{2})}{\Qg(\frac{2k+1}{2})}\qs_{n+k}-\frac{1}{2}\sum_{m=0}^{k-1}\frac{\Qg(\frac{2m+1}{2})u^{m+1}}{\Qg(\frac{2k+1}{2})(m+1)!}\qs_{n+k-m-1},
\end{align*}
and the Virasoro operators $L_m$ are given by
\begin{align}
\label{kdv-L-1}
L_{-1}=&\frac{t_0^2}{2}+\sum_{k\geq 0}t_{k+1}\diff{}{t_k}+\sum_{k\geq 1}(k+c_0)\qt_{k}\diff{}{\qt_{k-1}},\\
\label{kdv-L-m}
L_{m}=&\frac{1}{2}\sum_{k+l = m-1}\frac{\Qg(k+\frac{3}{2})\Qg(l+\frac{3}{2})}{\Qg(\frac{1}{2})\Qg(\frac{1}{2})}\frac{\qp^2}{\qp t_k\qp t_l}+\sum_{k\geq 0}\frac{\Qg(k+m+\frac{3}{2})}{\Qg(k+\frac{1}{2})}t_k\diff{}{t_{k+m}}\\\nonumber&+\sum_{k\geq 0}(k+c_0)\qt_k\diff{}{\qt_{k+m}}+\frac{1}{16}\qd_{m,0},\quad m\geq 0.
\end{align}
\end{Ex}


\section{The super tau-cover of the KdV hierarchy}
\label{sec5}
In this section, we give a deformation of the super tau-cover of the principal hierarchy, as it is given in Example \ref{kdv} and Example \ref{Vir-kdv}, of the one-dimensional Frobenius manifold. To this end, we are to employ a certain super extension of the Lax pair of the KdV hierarchy.

Let us first recall the Lax pair formalism of the KdV hierarchy. The Lax operator is given by 
\[L = \frac{\varepsilon^2}{2}\partial_x^2+u(x)-\lambda\] 
with the spectral parameter $\lambda$. The KdV hierarchy is the compatibility conditions of the linear systems
\begin{equation*}  
\begin{cases}
    L\Psi=0, \\
    \frac{\partial \Psi}{\partial t_n}=\frac{1}{2}B_n\Psi^\prime-\frac{1}{4}B_n^\prime\Psi.
\end{cases}
\end{equation*}
The functions $B_n$ are polynomials in $\lambda$ with coefficients being differential polynomials of $u(x)$, and they are determined in the following way. Consider the differential polynomials $R_n$ which are defined by the following recursive relation:
\begin{equation}
\label{rec-R}
(n+\frac{1}{2})\mathcal{P}_0 R_{n+1}=\mathcal{P}_1 R_n,\quad R_0=1,
\end{equation}
where the Hamiltonian operators $\mathcal{P}_0, \mathcal{P}_1$ are defined in \eqref{kdv-bho}, and the integral constant of $R_n$ for $n\geq 1$ is taken to be zero. Then the KdV hierarchy reads 
\[\diff{u}{t_n} = R_{n+1}^\prime,\quad n\ge 0.\] 
Here and in what follows we will use the notation $u_{t_n}:=\diff{u}{t_n}$. Denote
\[
b(\lambda) = \sum_{n\geq 0}\frac{\Gamma(n+\frac{1}{2})}{\lambda^{n+\frac{1}{2}}}R_n,\]
then the functions $B_n$ can be represented as
\[B_n= \frac{1}{\Gamma(n+\frac{3}{2})}(\lambda^{n+\frac{1}{2}}b(\lambda))_+,
\]
where the subscript $+$ means to take the non-negative part of the corresponding Laurent series in $\lambda$.

We introduce the super variables $\qs_n$ which satisfy the relations
\begin{equation}
\label{rec-eta}
\mathcal{P}_0\qs_{n+1}=\mathcal{P}_1\qs_n,\quad n\ge 0,
\end{equation}
and denote
\begin{equation}\label{gen-c}
c(\lambda)= -\sum_{n\geq 0}\frac{1}{\lambda^{n+1}}\qs_n,\quad
C_n= (\lambda^nc(\lambda))_-,
\end{equation}
where the subscript $-$ means to take the negative part of the corresponding Laurent series in $\lambda$.
Consider the following super extension of the Lax pair of the KdV hierarchy:
\begin{equation}
\label{def-superKdV}  
\begin{cases}
    L\Psi = 0, \\
    \frac{\partial \Psi}{\partial t_n} =\frac{1}{2}B_n\Psi^\prime-\frac{1}{4}B_n^\prime\Psi,\\
    \frac{\partial \Psi}{\partial \tau_n} =\frac{1}{2}C_n\Psi^\prime-\frac{1}{4}C_n^\prime\Psi.\\
\end{cases}
\end{equation}
The compatibility condition of these linear systems 
gives a deformation of the super tau-cover of the dispersionless KdV hierarchy, i.e. the principal hierarchy of the one-dimensional Frobenius manifold.
We are to write down the explicit formulae for the evolutions $u_{\tau_n}$, $(\qs_k)_{t_n}$ and $(\qs_k)_{\tau_n}$. To this end, we need the following lemmas and propositions.
\begin{Prop}
\label{u-taun}
We have the following evolutionary equations:
\begin{equation}
\label{u-tau0}
\frac{\partial u}{\partial \tau_n} = \qs_n^\prime,\quad n\ge 0.
\end{equation}
\end{Prop}
\begin{proof}
From the compatibility condition $(\Psi^{\prime\prime})_{\tau_n} = (\Psi_{\tau_n})^{\prime\prime}$ 
we arrive at the equation
\begin{equation}
\label{eq-eta}
\diff{u}{\tau_n} = \mathcal{P}_\lambda C_n,
\end{equation}
where 
\[\mathcal{P}_\lambda=\mathcal{P}_1-\lambda \mathcal{P}_0.\]
By using the equation \eqref{rec-eta}, we see that the right hand side of \eqref{eq-eta} equals $\qs_n^\prime$. The proposition is proved.
\end{proof}
\begin{Lem}
\label{c-tau0}
The function $c(\lambda)$ defined in \eqref{gen-c}
satisfies the following equation:
\begin{equation*}
\diff{c(\lambda)}{\tau_0} = \frac{1}{2}c(\lambda)c(\lambda)^\prime.
\end{equation*}
\end{Lem}
\begin{proof}
Since $C_0 = c(\lambda)$, the equation \eqref{eq-eta} with $n = 0$ leads to the relation
\begin{equation}
\label{eq-c}
\mathcal{P}_\lambda c(\lambda) = \qs_0^\prime.
\end{equation}
By applying $\diff{}{\tau_0}$ to both sides of \eqref{eq-c}, and by using the equation \eqref{u-tau0} and the fact that
$\diff{\qs_0^\prime}{\qt_0} =\frac{\partial^2 u}{\partial \tau_0^2} = 0$ we arrive at the equation
\begin{equation}
\label{temp1}
\mathcal{P}_\lambda \left(\diff{c(\lambda)}{\tau_0}\right)= -\qs_0^\prime c(\lambda)^\prime-\frac{1}{2}\qs_0^{\prime\prime}c(\lambda).
\end{equation}

Denote $Z=\diff{c(\lambda)}{\tau_0}-\frac{1}{2}c(\lambda)c(\lambda)^\prime$, then $Z$ is a power series in $1/\lambda$. We need to show that $Z = 0$. To this end, we rewrite the left hand side of \eqref{temp1} in the form
\begin{align}
\text{L.H.S.} =
&\frac{1}{2}c(\lambda)\left(\frac{3}{2}u^\prime(x)c(\lambda)^\prime+(u(x)-\lambda)c(\lambda)^{\prime\prime}+\frac{\vare^2}{8}c(\lambda)^{(4)}\right)\notag\\
&+\mathcal{P}_\lambda Z -\qs_0^\prime c(\lambda)^\prime.\label{eq-c-t0}
\end{align}
By using the equation \eqref{eq-c} and the fact that $c(\lambda)c(\lambda) = 0$ we obtain
\begin{equation*}
\frac{1}{2}c(\lambda)\left(\frac{3}{2}u^\prime(x)c(\lambda)^\prime+(u(x)-\lambda)c(\lambda)^{\prime\prime}+\frac{\vare^2}{8}c(\lambda)^{(4)}\right) = -\frac{1}{2}\qs_0'' c(\lambda).
\end{equation*}
Thus from the equations \eqref{temp1}, \eqref{eq-c-t0}
we deduce the identity
\begin{equation*}
\mathcal{P}_\lambda Z = 0.
\end{equation*}
Now by expanding $Z = \sum_{n\geq 0}z_n\lambda^{-n-1}$, we can show by induction that $z_n = 0$. Here we use the same argument as the one we used to prove Lemma \ref{Z-lem}. The lemma is proved.
\end{proof}

Now let us determine the evolutions of the super variables $\qs_m$ along $\tau_n$. Note that the compatibility conditions of the linear systems given in \eqref{def-superKdV} fix these evolutionary equations up to some constants, in what follows we take these constants to be zero.

\begin{Lem}
\label{eta0-taun}
The super variables $\qs_k$ satisfy the following equations:
\begin{equation}
\label{temp2}
\diff{\qs_0}{\tau_n} = \frac{1}{2}\sum_{i=0}^{n-1}\qs_i\qs_{n-i-1}^\prime.
\end{equation}
\end{Lem}
\begin{proof}
Introduce the operator:
\begin{equation}
\label{taumu}
\diff{}{\tau(\mu)} = \sum_{n\geq 0}\frac{1}{\mu^{n+1}}\diff{}{\tau_n}.
\end{equation}
From \eqref{u-tau0} it follows that $\diff{u}{\tau(\lambda)} = -c(\lambda)^\prime$, so we also have $\left(\diff{c(\lambda)}{\tau(\lambda)} \right)'= 0$. We can take the integration constant to be zero, and obtain the equation
$\diff{c(\lambda)}{\tau(\lambda)}=0$.
By applying $\diff{}{\tau(\lambda)}$ to both sides of the equation \eqref{eq-c}, we arrive at
\begin{equation*}
\frac{1}{2}c(\lambda)c(\lambda)^{\prime\prime} = \diff{\qs_0^\prime}{\tau(\lambda)}.
\end{equation*}
By taking the integration constant to be zero again we obtain
\begin{equation*}
\diff{\qs_0}{\tau(\lambda)} = \frac{1}{2}c(\lambda)c(\lambda)^{\prime},
\end{equation*}
which is just a reformulation of \eqref{temp2}. The lemma is proved.
\end{proof}

\begin{Prop}
\label{prop-tau}
The generating function satisfies the following evolutionary equations:
\begin{equation*}
\frac{\partial c(\lambda)}{\partial \tau_n} = \frac{1}{2}(C_nc(\lambda)^\prime+c(\lambda)C_n^\prime-(\lambda^nc(\lambda)c(\lambda)^\prime)_-).
\end{equation*}
\end{Prop}
\begin{proof}
The idea to prove this lemma is the same as the one that is used in the proof of Lemma \ref{c-tau0}. Denote
\[Z=\frac{\partial c(\lambda)}{\partial \tau_n} -\frac{1}{2}(C_nc(\lambda)^\prime+c(\lambda)C_n^\prime-(\lambda^nc(\lambda)c(\lambda)^\prime)_-),\]
then $Z$ is formal power series in $1/\lambda$. By applying $\diff{}{\tau_n}$ to both sides of the equation \eqref{eq-c} we obtain
\begin{equation}
\label{temp33}
C_n^\prime\qs_0^\prime+\frac{1}{2}C_n\qs_0^{\prime\prime}+\mathcal{P}_{\lambda}\left(Z-\frac{1}{2}(\lambda^nc(\lambda)c(\lambda)^\prime)_-\right) = \diff{\qs_0^\prime}{\tau_n}.
\end{equation}
From the definition of $c(\lambda)$ it follows that
\begin{equation*}
\lambda(\lambda^nc(\lambda)c(\lambda)^\prime)_--(\lambda^{n+1}c(\lambda)c(\lambda)^\prime)_- = \sum_{i=0}^{n-1}\qs_i\qs_{n-i-1}^\prime.
\end{equation*}
So by using Lemma \ref{eta0-taun} we have
\begin{align*}
&-\frac{1}{2}\mathcal{P}_{\lambda}(\lambda^nc(\lambda)c(\lambda)^\prime)_-\\ 
=& -\left(\frac{\lambda^n}{2}\mathcal{P}_1 \left(c(\lambda)c(\lambda)^\prime\right)\right)_-+\frac{\lambda}{2}\partial_x(\lambda^nc(\lambda)c(\lambda)^\prime)_-\\
=&-\left(\frac{\lambda^n}{2}\mathcal{P}_1 \left(c(\lambda)c(\lambda)^\prime\right)\right)_-+\frac{1}{2}\partial_x(\lambda^{n+1}c(\lambda)c(\lambda)^\prime)+\frac{1}{2}\partial_x\sum_{i=0}^{n-1}\qs_i\qs_{n-i-1}^\prime\\
=&-\left(\frac{\lambda^n}{2}\mathcal{P}_{\lambda} \left(c(\lambda)c(\lambda)^\prime\right)\right)_-+\diff{\qs_0^\prime}{\tau_n}.
\end{align*}
On the other hand, by substituting the result of Lemma \ref{c-tau0} into \eqref{temp1} we arrive at the identity
\begin{equation*}
-\frac{1}{2}\mathcal{P}_{\lambda}\left(c(\lambda)c(\lambda)^\prime\right) = -c(\lambda)^\prime\qs_0^\prime-\frac{1}{2}c(\lambda)\qs_0^{\prime\prime}.
\end{equation*}
Now by substituting the above two relations into \eqref{temp33} we arrive at $\mathcal{P}_{\lambda}Z = 0$, which implies $Z = 0$. The Theorem is proved.
\end{proof}

The following proposition is a reformulation of Proposition \ref{prop-tau}.
\begin{Cor}\mbox{}
The generating function $c(\lambda)$ satisfies the equation
	\begin{equation}
	\label{c-tau}
	\diff{c(\lambda)}{\tau(\mu)} = \frac{c(\lambda)c(\lambda)^\prime+c(\mu)c(\mu)^\prime-c(\mu)c(\lambda)^\prime-c(\lambda)c(\mu)^\prime}{2(\mu-\lambda)},
	\end{equation}
which yields the following evolutionary equations satisfied by $\sigma_k$:
	\begin{equation*}
	\diff{\qs_k}{\tau_n} = \frac{1}{2}\sum_{i=0}^{n-k-1}\qs_{i+k}\qs_{n-i-1}^\prime,\quad \diff{\qs_n}{\tau_k} = -\diff{\qs_k}{\tau_n},\quad k\le n.
	\end{equation*}

\end{Cor}

It remains to derive the evolutions of the generating function $c(\lambda)$ along $t_n$. For this we need the following lemma.

\begin{Lem}
The following zero-curvature equation holds true:
\begin{equation}
\label{zero-cur}
\diff{C_m}{t_n}-\diff{B_n}{\tau_m} = \frac{1}{2}(B_nC_m^\prime-C_mB_n^\prime). 
\end{equation}
\end{Lem}
\begin{proof}
This is just the compatibility condition $(\Psi_{\tau_m})_{t_n} = (\Psi_{t_n})_{\tau_m}$ of the linear systems given in  \eqref{def-superKdV}.
\end{proof}
\begin{Prop}\mbox{}
\begin{enumerate}
	\item The super variables $\qs_k$ satisfy the following
	evolutionary equations:
	\begin{equation*}
	\diff{\qs_k}{t_n} = \frac{1}{2}\sum_{i=0}^n\frac{\Gamma(n+\frac{1}{2}-i)}{\Gamma(n+\frac{3}{2})}(R_{n-i}\qs_{k+i}^\prime-R_{n-i}^\prime\qs_{k+i}).
	\end{equation*}
	\item The above equations can be represented in the form
	\begin{equation}
	\label{c-t}
	\diff{c(\lambda)}{t(\mu)} = \left(\frac{b(\mu)c(\lambda)^\prime-c(\lambda)b(\mu)^\prime}{2(\mu-\lambda)}\right)_{\frac{\lambda}{\mu},\lambda_-},
	\end{equation}
where
	\begin{equation}
	\label{tmu}
	\diff{}{t(\mu)} = \sum_{n\geq 0}\frac{\Qg(n+\frac{3}{2})}{\mu^{n+\frac{2}{3}}}\diff{}{t_n},
	\end{equation}
and the subscripts $(\ \, )_{\frac{\lambda}{\mu},\lambda_-}$mean that one first does the expansion
	\begin{equation*}\frac{1}{\mu-\lambda} = \frac{1}{\mu}\frac{1}{1-\frac{\lambda}{\mu}} = \frac{1}{\mu}\sum_{i\geq 0}\left(\frac{\lambda}{\mu}\right)^i,\end{equation*}
	and then take the negative part of the corresponding Laurent series in $\lambda$.
\end{enumerate}
\end{Prop}
\begin{proof}
We only need to prove the first statement, since the second one is just a reformulation of the first one. By taking the residue of both sides of the zero-curvature equation \eqref{zero-cur} with respect to $\lambda$ we obtain
\begin{equation}
\label{temp3}
\diff{\qs_m}{t_n} = -\frac{1}{2}\mathrm{res}_{\lambda = 0}(B_nC_m^\prime-C_mB_n^\prime).
\end{equation}
The result of the proposition then follows from the expressions 
\begin{equation*}
B_n = \sum_{i=0}^n\frac{\Gamma(n+\frac{1}{2}-i)}{\Gamma(n+\frac{3}{2})}R_{n-i}\lambda^i,\quad C_m = \sum_{i\geq 0}-\qs_{m+i}\lambda^{-i-1}.
\end{equation*} 
of $B_n, C_m$. The proposition is proved.
\end{proof}

The above results lead to the following theorem on the KdV hierarchy and its super extension.
\begin{Th}
The KdV hierarchy and its super extension read
\begin{align}
\label{u-t}
\diff{u}{t_n} &= R_{n+1}^\prime,\\
\label{u-tau}
\diff{u}{\tau_n} &= \qs_n^\prime,\\
\label{seg-t}
\diff{\qs_k}{t_n} &= \frac{1}{2}\sum_{i=0}^n\frac{\Gamma(n+\frac{1}{2}-i)}{\Gamma(n+\frac{3}{2})}(R_{n-i}\qs_{k+i}^\prime-R_{n-i}^\prime\qs_{k+i}),\\
\label{seg-tau}
\diff{\qs_k}{\tau_n} &= -\diff{\qs_n}{\tau_k} = \frac{1}{2}\sum_{i=0}^{n-k-1}\qs_{i+k}\qs_{n-i-1}^\prime,\quad \text{for $k \leq n$},
\end{align}
where the differential polynomials $R_k$ are defined in \eqref{rec-R}.
\end{Th}
\begin{proof}
The commutation relations $[\diff{}{t_m},\diff{}{t_n}]u = 0$ follow
form the original KdV hierarchy. By using the equation \eqref{eq-eta}, \eqref{zero-cur} and the equation
\begin{equation*}
\diff{u}{t_m} =\mathcal{P}_\lambda B_m,
\end{equation*}
we can show that $[\diff{}{t_m},\diff{}{\qt_n}]u = 0$.
The commutation relations $[\diff{}{\qt_m},\diff{}{\qt_n}]u = 0$ can be obtained straightforwardly from \eqref{u-tau} and \eqref{seg-tau}.

To show the commutation relations $[\diff{}{t_m},\diff{}{t_n}]\qs_{k} = 0$, it suffices to show that $\diff{}{t(\mu_1)}\diff{c(\ql)}{t(\mu_2)}$ is symmetric w.r.t. $\mu_1$ and $\mu_2$. To prove this, we first note that the commutation relations $[\diff{}{t_m},\diff{}{\qt_n}]u = 0$ are equivalent to
\begin{equation*}
\diff{b(\ql)}{\qt(\mu)}+\diff{c(\mu)}{t(\ql)} = 0,
\end{equation*}
from which it follows that
\begin{equation*}
\diff{}{t(\mu_1)}\diff{c(\ql)}{t(\mu_2)} = -\diff{}{t(\mu_1)}\diff{b(\mu_2)}{\qt(\ql)} = -\diff{}{\qt(\ql)}\diff{b(\mu_1)}{t(\mu_2)}.
\end{equation*}
The right hand side of the above identity is symmetric w.r.t. $\mu_1$ and $\mu_2$, which follows from the commutation relations $[\diff{}{t_m},\diff{}{t_n}]u = 0$. So we arrive at  $[\diff{}{t_m},\diff{}{t_n}]\qs_{k} = 0$.

The commutation relations $[\diff{}{\qt_m},\diff{}{\qt_n}]\qs_{k} = 0$ can be checked directly by using \eqref{c-tau}. Finally, the relation $[\diff{}{\qt_m},\diff{}{t_n}]\qs_{k} = 0$ can be checked by induction on $k$ using the recursion relation \eqref{rec-eta}. The theorem is proved.
\end{proof}

To construct the super tau-cover of the KdV hierarchy, we first 
define the differential polynomials $\Qo_{n,k}$ with vanishing constant terms such that $\Qo_{n,k}^\prime = \diff{R_{n+1}}{t_k}$. Then by introducing the one-point functions $f_n$, we have the tau-cover of the KdV hierarchy:
\begin{equation*}
\diff{f_k}{t_n} = \Qo_{k,n},\quad \diff{u}{t_n} = R_{n+1}^\prime,\quad
k, n\ge 0.
\end{equation*}
In order to introduce the odd flows, we must show that $\diff{R_{k+1}}{\qt_n}$ are total $x$-derivatives of some differential polynomials (c.f.  discussions given in Sec.\,\ref{sec4}). Due to the identities $\diff{R_{k+1}}{\qt_n} = \diff{\qs_{n+1}}{t_k}$ and \eqref{c-t}, we only need to prove that $b(\mu)c(\ql)^\prime$ is a total $x$-derivative.
\begin{Lem}
\label{temp4-1}
\begin{equation*}
b(\mu)c(\ql)^\prime = \left[\frac{1}{\mu-\ql}\left(\pi\frac{c(\ql)}{b(\mu)}-\frac{\qe^2}{8}b(\mu)\left(b(\mu)\left(\frac{c(\ql)}{b(\mu)}\right)^\prime\right)^\prime\right)^\prime\right]_{\frac{\ql}{\mu},\ql_-},
\end{equation*}
\end{Lem}
\begin{proof}
It is easy to derive the following equation from \eqref{rec-R}:
\begin{equation}
\label{eq-b}
(u(x)-\ql)b(\ql)^2+\frac{\qe^2}{8}(2b(\ql)b(\ql)^{\prime\prime}-(b(\ql)^\prime)^2) = -\pi.
\end{equation}
We first apply the operator \eqref{tmu} to both sides of \eqref{eq-b}, then multiply the resulting equation by $b(\ql)$ and use \eqref{eq-b} 
again to obtain the following relation:
\begin{equation}\label{bc-prime}
b(\mu)c(\ql)^\prime = \left[\left(\frac{\pi}{b(\mu)^2}-\frac{\qe^2}{8}\frac{\mathscr D(\mu)}{b(\mu)^2}\right)\frac{b(\mu)c(\ql)^\prime-c(\ql)b(\mu)^\prime}{\mu-\ql}\right]_{\frac{\ql}{\mu},\ql_-},
\end{equation}
here the operator $\mathscr D(\mu)$ is defined by
\begin{equation*}
\mathscr D(\mu) = b(\mu)^2\qp_x^2-b(\mu)b(\mu)^\prime\qp_x+(b(\mu)^\prime)^2-b(\mu)b(\mu)^{\prime\prime}.
\end{equation*}
The lemma then follows from \eqref{bc-prime}.
\end{proof}

\begin{Th}\mbox{}
\begin{enumerate}
	\item
There exist differential polynomials $\Phi_{k}^n$ such that $\diff{R_{k+1}}{\qt_n} = (\Phi_{k}^n)^\prime$.
\item
The following equations:
\begin{align}
\label{kdvf-t}
\diff{f_k}{t_n} &= \Qo_{k,n};\\
\label{ff-tau}
\diff{f_k}{\qt_n} &= \Phi_k^n,
\end{align}
together with the flows \eqref{u-t}--\eqref{seg-tau} give the super tau-cover of the KdV hierarchy.
\end{enumerate}
\end{Th}
\begin{proof}
The first statement follows from \eqref{c-t} and Lemma \ref{temp4-1}. The commutativity of the flows given in the second statement follows from the relations $f_k^\prime = R_{k+1}$. The theorem is proved.
\end{proof}

From these explicit formulae for the super tau-cover of the KdV hierarchy, it is not hard to check that its dispersionless limit is indeed the super tau-cover of the dispersionless KdV hierarchy, which is described in Sec.\,\ref{sec3} and Sec.\,\ref{sec4}. From the expressions of the first two odd flows $\diff{}{\qt_0}$ and $\diff{}{\qt_1}$, we can recover the bihamiltonian structure of the KdV hierarchy given by the local functionals
\begin{equation*}
P_0 = \frac{1}{2}\int \qth\qth^1;\quad P_1 = \frac{1}{2}\int \left(u\qth\qth^1+\frac{\qe^2}{8}\qth\qth^3\right).
\end{equation*}
It remains to show that this super tau-cover of the KdV hierarchy is exactly the deformation of the super tau-cover of the dispersionless KdV hierarchy which admits Virasoro symmetries, and these symmetries act linearly on the tau function of the KdV hierarchy. 

To write down the Virasoro symmetries of the super tau-cover of the KdV hierarchy, we modify the Virasoro operators presented in Example \ref{Vir-kdv} by inserting the dispersion parameter $\qe$ as follows:
\begin{align*}
L_{-1}&=\frac{t_0^2}{2\qe^2}+\sum_{k\geq 0}t_{k+1}\diff{}{t_k}+\sum_{k\geq 0}(k+c_0)\qt_{k+1}\diff{}{\qt_k},\\
L_{m}&=\frac{\qe^2}{2}\sum_{k+l = m-1}\frac{\Qg(k+\frac{3}{2})\Qg(l+\frac{3}{2})}{\Qg(\frac{1}{2})\Qg(\frac{1}{2})}\frac{\qp^2}{\qp t_k\qp t_l}+\sum_{k\geq 0}\frac{\Qg(k+m+\frac{3}{2})}{\Qg(k+\frac{1}{2})}t_k\diff{}{t_{k+m}}\\\nonumber&+\sum_{k\geq 0}(k+c_0)\qt_k\diff{}{\qt_{k+m}}+\frac{1}{16}\qd_{m,0},\quad m\geq 0.
\end{align*}
Note that the tau function $Z_{\mathrm{KdV}}$ of the KdV hierarchy satisfies the relation
\[\qe \frac{\qp\log Z_{\mathrm{KdV}}}{\qp t_n}=f_n,\quad n\ge 0,\]
and the requirement of linear actions of the Virasoro symmetries suggests that these symmetries are given by the flows
\[\frac{\qp Z_{\mathrm{KdV}}}{\qp s_m}=L_mZ_{\mathrm{KdV}},\quad m\ge -1.\]
Thus the actions of these Virasoro symmetries on the one-point functions are given by
\begin{align}
\label{fkdv-s-1}
\diff{f_n}{s_{-1}}=&\frac{t_0}{\qe}\qd_{n,0}+f_{n-1}+\sum_{k\geq 0}t_{k+1}\diff{f_n}{t_k}+\sum_{k\geq 0}(k+c_0)\qt_{k+1}\diff{f_n}{\qt_k}\\
\label{fkdv-s-m}
\diff{f_n}{s_m}=&\frac{\qe^2}{2}\sum_{k+l = m-1}\frac{\Qg(k+\frac{3}{2})\Qg(l+\frac{3}{2})}{\Qg(\frac{1}{2})\Qg(\frac{1}{2})}\frac{\qp^2 f_n}{\qp t_k\qp t_l}\\
&+\qe\sum_{k+l = m-1}\frac{\Qg(k+\frac{3}{2})\Qg(l+\frac{3}{2})}{\Qg(\frac{1}{2})\Qg(\frac{1}{2})}\diff{f_n}{t_k}f_l+\frac{\Qg(n+m+\frac{3}{2})}{\Qg(n+\frac{1}{2})}f_{n+m}\notag\\
&+\sum_{k\geq 0}\frac{\Qg(k+m+\frac{3}{2})}{\Qg(k+\frac{1}{2})}t_k\diff{f_n}{t_{k+m}}+\sum_{k\geq 0}(k+c_0)\qt_k\diff{f_n}{\qt_{k+m}},\quad m\geq 0.\notag
\end{align}
The action of the Virasoro symmetries on the other unknown functions are then given by 
\begin{equation}
\label{fkdv-v}
\diff{u}{s_m}:=\qe\diff{}{t_0}\diff{f_0}{s_m},\quad\diff{\qs_n}{s_m}:=\qe\diff{}{\qt_n}\diff{f_0}{s_m}.
\end{equation}
The following theorem shows that the above defined flows indeed give the Virasoro symmetries of the super tau-cover of the KdV hierarchy.
\begin{Th}
The flows \eqref{u-t}--\eqref{seg-tau}, \eqref{kdvf-t}, \eqref{ff-tau} of the super tau-cover of the KdV hierarchy and the flows \eqref{fkdv-s-1}--\eqref{fkdv-v} satisfy the following commutation relations:
\begin{equation*}
\left[\diff{}{t_k},\diff{}{s_m}\right]=0,\quad \left[\diff{}{\tau_k},\diff{}{s_m}\right]=0,\quad \left[\diff{}{s_n},\diff{}{s_m}\right]=(m-n)\diff{}{s_{n+m}},
\end{equation*}
where $m, n\ge -1,\, k\ge 0$.
\end{Th}
\begin{proof}
The proof is a straightforward computation, using a similar induction procedure as we do in the proof of Theorem \ref{super-3}. We omit the details here.
\end{proof}


\section{Conclusion}
\label{sec6}
For an arbitrary Frobenius manifold $M$, we construct the super tau-cover of the associated principal hierarchy in Theorems \ref{super-1}, \ref{super-2} and \ref{super-3}. We also illustrate the existence of a deformation of the super tau-cover of the dispersionless KdV hierarchy which possesses Virasoro symmetries acting linearly on the tau function of the KdV hierarchy. 

According to \cite{dubrovin2001normal}, the KdV hierarchy is uniquely determined by the linearizable Virasoro symmetries.
Our construction of the super extension of the KdV hierarchy given in Sec.\,\ref{sec5} implies that this deformation of the super tau-cover of the dispersionless KdV hierarchy with linearizable Virasoro symmetries is unique up to the addition of some constants to
the flows $\diff{\sigma_k}{\tau_n}$. If we require that these flows belong to $\hat{\mathcal{A}}^2(M)$, then they are also uniquely fixed.

In general, we have the following conjecture.
\begin{Conj}
For a semisimple Frobenius manifold $M$, there exists a unique homogeneous deformation of the super tau cover of its principal hierarchy
such that it possesses Virasoro symmetries which are induced from the following Virasoro symmetries of its tau function $Z_M$:
\begin{equation}
\diff{Z_{M}}{s_m}=L_m Z_M,\quad m\ge -1,
\end{equation}
where $L_m$ is given in \eqref{L}.
\end{Conj}

Note that the first two odd flows of the super tau cover are exactly the bihamiltonian structure of the even flows,
so if the above conjecture holds true, then we have the following corollary.
\begin{Cor}
For a semisimple Frobenius manifold $M$, there exists a unique deformation of its principal hierarchy such that
it possesses a differential polynomial bihamiltonian structure and a set of Virasoro symmetries which act linearly on its tau function. 
\end{Cor}
More precisely, the above corollary shows that the integrable hierarchy satisfying all four axioms of Dubrovin-Zhang theory 
given in \cite{dubrovin2001normal} does exist.

Besides the above applications in the theory of Frobenius manifolds, one can also generalize the notion of super tau cover
to more general integrable systems, such as the ones with nonlocal bihamiltonian structures. For example, we expect that
there also exists super tau cover for the Sawada-Kotera hierarchy \cite{sawada1974method}, Kaup-Kuperschmidt hierarchy \cite{kaup1980inverse}, and other Drinfeld-Sokolov hierarchies \cite{drinfel1985lie, liu2020drinfeld}
associated to twisted affine Lie algebras.

\medskip

\noindent Si-Qi Liu,

\noindent Department of Mathematical Sciences, Tsinghua University \\ 
Beijing 100084, P.R.~China\\
liusq@tsinghua.edu.cn
\medskip

\noindent Zhe Wang,

\noindent Department of Mathematical Sciences, Tsinghua University \\ 
Beijing 100084, P.R.~China\\
zhe-wang17@mails.tsinghua.edu.cn
\medskip

\noindent Youjin Zhang,

\noindent Department of Mathematical Sciences, Tsinghua University \\ 
Beijing 100084, P.R.~China\\
youjin@tsinghua.edu.cn

\end{document}